\pdfoutput=1
\RequirePackage[l2tabu, orthodox]{nag}

\documentclass[reqno]{amsart}

\usepackage{lmodern}
\usepackage[T1]{fontenc}
\usepackage[utf8]{inputenc}
\usepackage{microtype}

\usepackage[english]{babel}

\usepackage{amsmath,amssymb,amsthm,mathrsfs,latexsym,
mathtools,mathdots,booktabs,array,enumerate,tikz,bm,url,comment}
\usepackage[centertableaux]{ytableau}
\usepackage{extarrows}
\usepackage{hyperref}

\usepackage{tikz-cd}

\usetikzlibrary{arrows,positioning}
\usepackage[capitalize,noabbrev]{cleveref}

\usepackage{todonotes}
\presetkeys%
    {todonotes}%
    {inline,backgroundcolor=yellow}{}

\newcommand*\circled[1]{\tikz[baseline=(char.base)]{\node[shape=circle,draw,inner sep=2pt] (char) {#1};}}

\newtheorem{theorem}{Theorem}
\newtheorem{proposition}[theorem]{Proposition}
\newtheorem{lemma}[theorem]{Lemma}
\newtheorem{corollary}[theorem]{Corollary}

\theoremstyle{definition}
\newtheorem{definition}[theorem]{Definition}
\newtheorem{example}[theorem]{Example}
\newtheorem{remark}[theorem]{Remark}

\newcommand{\oeis}[1]{\href{http://oeis.org/#1}{#1}}

\newcommand{\defin}[1]{\emph{#1}}

\newcommand{\thsup}{\textnormal{th}}

\newcommand{\setP}{\mathbb{N_+}}
\newcommand{\setN}{\mathbb{N}}
\newcommand{\setZ}{\mathbb{Z}}

\newcommand{\setC}{\mathbb{C}}

\newcommand{\xvec}{\mathbf{x}}

\newcommand{\cryse}{\tilde{e}}
\newcommand{\crysf}{\tilde{f}}
\newcommand{\cryss}{\tilde{s}}

\newcommand{\powerSum}{\mathrm{p}}
\newcommand{\elementaryE}{\mathrm{e}}
\newcommand{\monomial}{\mathrm{m}}
\newcommand{\schurS}{\mathrm{s}}
\newcommand{\completeH}{\mathrm{h}}
\newcommand{\macdonaldH}{\mathrm{\tilde H}}
\newcommand{\LLT}{\mathrm{LLT}}
\newcommand{\macdonaldE}{\mathrm{E}}
\newcommand{\hallLittlewoodT}{\mathrm{Q}'}

\newcommand{\nuvec}{{\boldsymbol\nu}}

\newcommand{\symS}{S}

\newcommand{\SSYT}{\mathrm{SSYT}}
\newcommand{\SYT}{\mathrm{SYT}}
\newcommand{\COF}{\mathrm{COF}}

\newcommand{\Des}{\mathrm{Des}}

\newcommand{\NEpath}[4]{
	\fill[white!25]  (#1) rectangle +(#2,#3);
	\fill[fill=white]
	(#1)
	\foreach \dir in {#4}{
		\ifnum\dir=0
		-- ++(1,0)
		\else
		-- ++(0,1)
		\fi
	} |- (#1);
	\draw[help lines] (#1) grid +(#2,#3);
    \draw[dashed] (#1) -- +(#3,#3);
	\coordinate (prev) at (#1);
	\foreach \dir in {#4}{
		\ifnum\dir=0
		\coordinate (dep) at (1,0);
		\else
		\coordinate (dep) at (0,1);
		\fi
		\draw[line width=1.0pt,-stealth] (prev) -- ++(dep) coordinate (prev);
	};
}

\DeclareMathOperator{\stat}{stat}

\DeclareMathOperator{\rw}{rw}
\DeclareMathOperator{\cw}{cw}
\DeclareMathOperator{\rev}{rev}
\DeclareMathOperator{\length}{\ell}

\DeclareMathOperator{\weight}{wt}

\DeclareMathOperator{\inv}{inv}
\DeclareMathOperator{\mininv}{mininv}
\DeclareMathOperator{\maj}{maj}

\DeclareMathOperator{\charge}{charge}

\DeclareMathOperator{\leg}{leg}
\DeclareMathOperator{\ins}{ins}
\DeclareMathOperator{\rec}{rec}

\newcommand{\multiset}[2]{\ensuremath{\left(\kern-.3em\left(\genfrac{}{}{0pt}{}{#1}{#2}\right)\kern-.3em\right)}}
\newcommand{\qbinom}{\genfrac{[}{]}{0pt}{}}

\title{Cyclic sieving, skew Macdonald polynomials and Schur positivity}

\author{Per Alexandersson}
\author{Joakim Uhlin}
\address{Dept. of Mathematics, Royal Institute of Technology, SE-100 44 Stockholm, Sweden}
\email{per.w.alexandersson@gmail.com, joakim\_uhlin@hotmail.com}

\begin{document}

\begin{abstract}
When $\lambda$ is a partition, the specialized non-symmetric Macdonald polynomial
$\macdonaldE_{\lambda}(\xvec;q;0)$ is symmetric and related to a modified Hall--Littlewood polynomial.
We show that whenever all parts of the integer partition $\lambda$ is a multiple of $n$, 
the underlying set of fillings exhibit the cyclic sieving 
phenomenon (CSP) under a cyclic shift of the columns.
The corresponding CSP polynomial is given by $\macdonaldE_{\lambda}(\xvec;q;0)$.
In addition, we prove a refined cyclic sieving phenomenon where the content of the fillings is fixed.
This refinement is closely related to an earlier result by B.~Rhoades.

We also introduce a skew version of $\macdonaldE_{\lambda}(\xvec;q;0)$.
We show that these are symmetric and Schur-positive via a variant of 
the Robinson--Schenstedt--Knuth correspondence and we also describe crystal raising- and lowering operators 
for the underlying fillings.
Moreover, we show that the skew specialized non-symmetric Macdonald polynomials 
are in some cases vertical-strip LLT polynomials.
As a consequence, we get a combinatorial Schur expansion of a new family of LLT polynomials.
\end{abstract}

\maketitle

\setcounter{tocdepth}{1}
\tableofcontents

\section{Introduction}

The cyclic sieving phenomenon (CSP), introduced by V. Reiner, D. Stanton and D. White~\cite{ReinerStantonWhite2004},
is currently an active research topic.
In this article, we provide families of cyclic sieving on tableaux related to 
certain specializations of non-symmetric Macdonald polynomials.
This settles an earlier conjecture by the authors presented in \cite{Uhlin2019}.
The non-symmetric Macdonald polynomials are in our case closely 
related to the transformed Hall--Littlewood functions and Kostka--Foulkes polynomials, 
previously studied in the CSP context by B.~Rhoades~\cite{Rhoades2010b}.
The family of polynomials we study is the specialization of the non-symmetric Macdonald polynomials
$\macdonaldE_\lambda(x_1,\dotsc,x_m;q,t)$ when $\lambda$ is an integer partition and $t=0$.
They can be defined as a weighted sum over certain fillings of the Young diagram $\lambda$.
We denote this set of fillings $\COF(\lambda,m)$, which is defined further down.

\subsection{Main results}

For an integer partition $\lambda$, we let $n\lambda \coloneqq (n\lambda_1,\dotsc,n \lambda_\ell)$.
We show that there is a natural action $\phi$
on the fillings $\COF(n\lambda,m)$ each block of $n$ consecutive columns 
is cyclically rotated one step.
Consequently $\phi$ generates a $C_n$-action on $\COF(n\lambda,m)$.
In \cref{thm:mainCSP}, we prove that for every integer partition $\lambda$, $m \in \setP$
and $n \in \setP$, the triple
\begin{equation}\label{eq:cspIntro}
 \left( \COF(n \lambda, m), \langle \phi \rangle , \macdonaldE_{n\lambda}(1^m,q,0) \right)
\end{equation}
exhibits the cyclic sieving phenomenon. 
Moreover,
as $\lambda$ is held fixed and $n=1,2,3,\dotsc$, this family is a \emph{Lyndon-like family}, 
a notion by P. Alexandersson, S. Linusson and S. Potka~\cite{AlexanderssonLinussonPotka2019} 
(see also \cite{Gorodetsky2019}) meaning that fixed points 
under the group action are in natural bijection
with smaller instances of the combinatorial objects.
When $\lambda = (1)$, this phenomenon reduces to a classical cyclic sieving phenomenon on
words of length $n$ in the alphabet $[m]$, see \cref{ex:CSPonwords} below.
A skew version of \eqref{eq:cspIntro} is given in \cref{thm:mainCSPSkew}.

We also prove a refined cyclic sieving phenomenon.
Let $\COF(n\lambda,\nu)$ denote the set of coinversion-free fillings with shape $n\lambda$ and content $\nu$.
In \cref{thm:refinedMacdonaldCSP}, we show that
\begin{equation}\label{eq:refinedIntro}
 \left( \COF(n \lambda, \nu), \langle \phi \rangle, [\monomial_\nu]\macdonaldE_{n \lambda}(\xvec,q,0) \right)
\end{equation}
exhibits the cyclic sieving phenomenon. When $\lambda = (n)$, we recover the 
cyclic sieving phenomenon on words of length $n$ with content $\nu$,
and 
\[
[\monomial_\nu]\macdonaldE_{(n)}(\xvec,q,0) = \qbinom{n}{\nu}_q,
\]
a $q$-multinomial coefficient.
We remark that if we take $\lambda=(1^k)$, 
$[\monomial_\nu]\macdonaldE_{n\lambda}(\xvec,q,0)$ in~\eqref{eq:refinedIntro} 
can be seen as a $q$-analogue of $n$-tuples of $k$-subsets of $[m]$ with content $\nu$.
We expect that many properties of the usual major index extend to this generalization.

Finally in \cref{sec:skew} we introduce a skew version of $\macdonaldE_{\lambda}(\xvec,q,0)$
and prove that these are symmetric and Schur positive.
We provide an explicit Schur expansion using a generalization of 
charge in \cref{thm:skewEInSchurExpansion}.
As an application, in \cref{thm:lltNewFormula} we obtain a combinatorial 
Schur-expansion of a certain family of \emph{vertical-strip LLT polynomials},
which has not been considered before.
Combining \cref{thm:lltNewFormula} and \cref{thm:skewEInSchurExpansion}, we have the following main result.
\begin{theorem}\label{thm:lltFromCharge}
Let $\lambda/\mu$ be a skew shape with $\ell$ rows,
such that no column contains no more than two boxes. 
Let $\alpha$ be the composition defined as $\alpha_i = \lambda_i-\mu_i$
and let $\nuvec$ be the tuple of skew shapes such that 
$\nu_{j}$ is the vertical strip $1^{\lambda_j} /1^{\mu_j}$.
Then 
\[
 \LLT_{\nuvec}(\xvec;q) = 
 q^{\mininv(\nuvec)} \sum_{\rho \vdash |\lambda'/\mu'|}  \schurS_{\rho'}(\xvec) 
 \sum_{T \in \SSYT(\rho,\alpha)}  q^{\charge_{\mu}(T)}
\]
where $\charge_{\mu}$ is a natural generalization of the charge statistic,
and $\mininv(\nuvec)$ is a simple statistic that only depend on the tuple $\nuvec$.
\end{theorem}

The paper is structured as follows. 
In \cref{sec:prelim}, we define the cyclic sieving phenomenon and give a brief overview of 
the relevant symmetric functions.
In \cref{sec:csp1}, we give a proof of the CSP in \eqref{eq:cspIntro},
and in \cref{sec:cspRefined} we prove \eqref{eq:refinedIntro}.
In \cref{sec:skew} we introduce the skew specialized Macdonald polynomials
and give the Schur expansion of these. 
We then prove a result in \cref{sec:llt} which implies \cref{thm:lltFromCharge}.
%

We note that some of the results in this paper are based on earlier work done in the 
second author's master's thesis~\cite{Uhlin2019}.

\section{Preliminaries}\label{sec:prelim}

\subsection{Partitions and compositions}
\begin{definition}
Let $n$ and $r$ be natural numbers. A \defin{weak composition} $\lambda$ of $n$
into $r$ \defin{parts} is defined to be an $r$-tuple $\lambda = (\lambda_1 ,\dotsc,\lambda_r)$ 
of non-negative integers such that $\lambda_1+\dotsb+\lambda_r = n$. 
We say that the numbers $\lambda_1,\dotsc, \lambda_r$ are the \defin{parts} of $\lambda$. 
If all parts of $\lambda$ are positive, we say that $\lambda$ is a \defin{composition},
and we write $\lambda \vDash n$.
If $\lambda$ has multiple parts of the same size, we may suppress them using exponents. 
As an example, $(7, 7, 0, 1, 1, 1, 4, 4, 4, 4)$ may be expressed as
$(7^2, 0, 1^3, 4^4)$. 
We write $m_j(\lambda)$ for the number of parts of $\lambda$
equal to $j$ and $n\lambda \coloneqq (n\lambda_1,\dotsc,n\lambda_\ell)$ for $n\in \setN$.
If further $\lambda_1 \geq  \lambda_2 \geq \dotsb \geq \lambda_r$, then $\lambda$ is a \defin{partition} 
of $n$, and denote this by $\lambda \vdash n$. 
The \defin{parts} of $\lambda$ are the positive entries of $\lambda$. 
The \defin{length} of $\lambda$ is the number of parts and is denoted $\ell(\lambda)$. 
We identify partitions that only differ by trailing zeros, 
so $(4,2,2,1,0,0,0)$=$(4,2,2,1,0)$=$(4,2,2,1)$ as partitions. 
There is one unique partition of $0$, namely $\emptyset$ which is referred to as the \defin{empty partition}.
\end{definition}

Note that in some cases, the word \emph{parts} is ambiguous. When $\lambda$ is a weak
composition, a part can be zero whereas when $\lambda$ is a partition, a part must be a
positive integer. 
This conflicting terminology is unfortunately very standard, see e.g.~\cite{StanleyEC2}.

\subsection{Semistandard Young tableau}

\begin{definition}
Let $\lambda=(\lambda_1, \dots, \lambda_r) \vdash n$. 
The \defin{Young diagram} of $\lambda$ is defined as the set $\{(i,j) \in \setZ^2 : 1 \leq i \leq \lambda_j \}$. 
Geometrically, we think of this diagram as a set of $n$ boxes with $r$ 
left-justified rows and $i$ boxes in row $i$, counting from top to bottom, 
starting from row $1$ and column $1$. 
The box in position $(i, j)$ is the box in the $i^\thsup$ row 
and $j^\thsup$ column. 
We use the notation $\lambda$ to both refer to the partition and to the Young diagram described by $\lambda$.
	
Define the \defin{conjugate} of $\lambda$, denoted $\lambda'$, to be
the Young diagram obtained by transposing $\lambda$. 
Geometrically, the conjugate of $\lambda$ is obtained by reflecting $\lambda$ across the line $y = -x$. 
We write $\lambda' = (\lambda_1', \dots, \lambda_l')$. 
If $\lambda$ is a partition on the form $\lambda = (a^b)$, 
then  $\lambda$ is a \defin{rectangular Young diagram}.
\end{definition}
Throughout this article, all the diagrams are displayed in English notation, using matrix coordinates,
with a few exceptions in \cref{sec:llt}.

\begin{figure}[!ht]
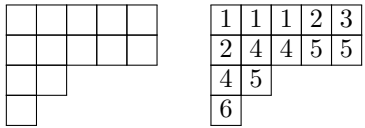

	\ytableausetup{boxsize=1.1em}
	\[\ytableaushort{\;\;\;\;\;, \;\;\;\;\;,\;\;,\;} \qquad \ytableaushort{11123,24455, 45, 6}\]
	\caption{To the left: A Young diagram of shape $\lambda=(5,5,2,1)$. 
	To the right: A semistandard Young tableau of shape $\lambda$.}
	\label{fig: SSYT example}
\end{figure}

\begin{definition}\label{def:SSYT}
Let $\lambda$ be a Young diagram. A \defin{filling} of $\lambda$ is a map $\lambda \to \setP$ and a \defin{semistandard Young tableau} is a 
filling of $\lambda$ with positive integers such that in each row the entries 
are weakly increasing and in each row the entries are strictly increasing. 
The set of all semistandard Young tableaux of shape $\lambda$ is denoted $\SSYT(\lambda)$.

Let $T$ be a semistandard Young tableau. Define the \defin{reading word} of $T$, 
denoted $\rw (T)$, as the word obtained by reading the entries $T$ from the 
bottom row to the top row and in each row from left to right. 
For example, the semistandard Young tableau 
in \cref{fig: SSYT example} has reading word $\rw(T)=6452445511123$. 
We let $\xvec^T \coloneqq \prod_j x_j^{m_j(T)}$ where $m_j(T)$ is the number of entries 
in $T$ equal to $j$.
The semistandard Young tableau $T$ in \cref{fig: SSYT example} gives $\xvec^T=x_1^3x_2^2x_3x_4^3x_5^2x_6$.
\end{definition}
There are several equivalent ways to define the Schur functions 
but the following is the most useful for our purposes.
We let the \defin{Schur function} indexed by the integer partition $\lambda$ be defined as
\[
\schurS_\lambda(\xvec) \coloneqq \sum_{T \in \SSYT(\lambda)} \xvec^T.
\]

\begin{example}
Let $\lambda=(7,6,3,3,3,1)$ be a partition. 
Then $|\lambda|=23$, the conjugate partition, $\lambda'$, is given by $(6,5,5,2,2,2,1)$ 
and $3\lambda=(21,18,9,9,9,3)$. 
Furthermore, $m_1(\lambda)=1$, $m_2(\lambda)=0$ and $m_3(\lambda)=3$.
\end{example}

\subsection{Burge words and RSK}\label{sec:burgeRSK}

The Robinson--Schenstedt--Knuth correspondence (RSK) is a famous combinatorial bijection with 
many different applications \cite{StanleyEC2,Krattenthaler2006}. 
The version we use in this paper is a bijection between pairs of certain biwords and 
pairs of semistandard Young tableaux. 
We note that the biwords we consider are not lexigraphically ordered, which is otherwise typical.

\begin{definition}\label{def:burge word}
	A \defin{Burge word} is a two-line array with positive integers
	\[W=\begin{pmatrix}
	i_1 & i_2 & \dotsb & i_m\\
	j_1 & j_2 & \dotsb & j_m
	\end{pmatrix} \]
	sorted primarily increasingly in the first row and secondarily on the second row decreasingly. 
	Furthermore, all columns are unique. As an example, 
	$\left(\begin{smallmatrix} 1 & 1 & 2 & 3 & 3 & 3 & 3 & 5 & 6 & 6 & 6 \\ 
	3 & 1 & 2 & 6 & 4 & 3 & 2 & 4 & 5 & 3 & 1 \end{smallmatrix}\right)$ is a Burge word.
	A pair $(i_c, j_c)$ is called a \defin{biletter}. The first row of $W$ is called the \defin{recording word}
	and the second row of the biword is called the \defin{charge word} --- the 
	reason for this terminology will be apparent in \cref{prop:majAspostfixCharge}.
\end{definition}
We use the same row insertion bumping algorithm as the standard RSK on biwords, which we assume the readers are familiar with. 
Properties of our version of RSK is the third variant described by C.~Krattenthaler~\cite{Krattenthaler2006}.


\ytableausetup{boxsize=0.9em}
\begin{table}[ht!]
	\begin{tabular}{l|l|l|l|l|l|l|l|l}
	\text{Inserted biletter}	\rule{0pt}{4ex} & $\binom{1}{4}$ & $\binom{1}{1}$ & $\binom{2}{3}$ & $\binom{2}{2}$ & $\binom{4}{5}$ & $\binom{5}{4}$ & $\binom{5}{3}$ & $\binom{5}{1}$
	\\ \hline
	P & \ytableaushort{4} & \ytableaushort{1,4} & \ytableaushort{13,4} & \ytableaushort{12,3,4} & \ytableaushort{125,3,4} & \ytableaushort{124,35,4} & \ytableaushort{123,34,45} & \ytableaushort{\none, 113,24,35,4,\none}
	\\ \hline
	Q & \ytableaushort{1} & \ytableaushort{1,1} & \ytableaushort{12,1} & \ytableaushort{12,1,2} & \ytableaushort{124,1,2} & \ytableaushort{124,15,2} & \ytableaushort{124,15,25} & \ytableaushort{\none, 124,15,25,5,\none}
	\end{tabular}

\caption{Computing the image of a Burge word 
under RSK via a sequence of row insertions.}
\label{table: RSK}
\end{table}

\begin{proposition}\label{prop:RSK}
	The RSK-algorithm yields a bijection between Burge words and pairs of fillings $(P,Q)$ of the same shape
	such that the \defin{insertion tableau} $P$ is semistandard and the \defin{recording tableau} $Q$ has the property that $Q^t$ is semistandard.
\end{proposition}
As an example of \cref{prop:RSK}, the procedure in \cref{table: RSK} shows that we have the following correspondence.
\begin{align*}
\begin{pmatrix}
1 & 1 & 2 & 2 & 4 & 5 & 5 & 5\\
4 & 1 & 3 & 2 & 5 & 4 & 3 & 1
\end{pmatrix} 
\xlongrightarrow{RSK}
\left(
\ytableaushort{113,24,35,4}\; ,\quad
\ytableaushort{124,15,25,5}
\right).
\end{align*}

\subsection{\texorpdfstring{$q$}{q}-analogues}

A $q$-analogue of a certain expression is a rational function in the variable $q$ 
from which we can obtain the original expression by letting $q \to 1$.
\begin{definition}
Let $n \in \setN$. Define the \defin{$q$-analogue of $n$} as
$[n]_q \coloneqq 1+q+\dotsb+q^{n-1}$.
Furthermore, define the \defin{$q$-factorial of $n$} as
$
 [n]_q! \coloneqq [n]_q[n-1]_q \dotsm [1]_q
$.
Lastly, the \defin{$q$-binomial coefficient} are defined as  
\[
	\qbinom{n}{k}_q \coloneqq 
	\dfrac{[n]_q!}{[n-k]_q![k]_q!}  \text{ if } n \geq k\geq 0, \text{ and $0$ otherwise.} 
\]
\end{definition}
\begin{theorem}[$q$-Lucas theorem, see e.g \cite{Sagan1992}]\label{eq:q-Lucas}
	Let $n, k\in \setN$. Let $n_1, n_0, k_1, k_0$ be the unique 
	natural numbers satisfying $0 \leq n_0, k_0 \leq d-1$ and $n=n_1d+n_0$, $k=k_1d+k_0$. Then 
	\[
	\qbinom{n}{k}_q \equiv \binom{n_1}{k_1}\qbinom{n_0}{k_0}_q \pmod{\Phi_d(q)}
	\]
	where $\Phi_d(q)$ is the $d^\thsup$ cyclotomic polynomial. In particular, we have
	\begin{equation}
	\qbinom{n}{k}_\xi = \binom{n_1}{k_1}\qbinom{n_0}{k_0}_\xi.
	\end{equation}
	if $\xi$ is a primitive $d^\thsup$ root
\end{theorem}
\cref{eq:q-Lucas} will be used in later sections.

\subsection{Cyclic sieving}\label{ssec:csp}

\begin{definition}[Cyclic sieving, see \cite{ReinerStantonWhite2004}]
Let $X$ be a set of combinatorial objects and $C_n$ be the cyclic group of order $n$
acting on $X$. Let $f(q)\in \setN[q]$ be a polynomial with non-negative integer coefficients.
We say that the triple $(X,C_n,f(q))$ \defin{exhibits the cyclic sieving phenomenon, (CSP)}
if for all $d \in \setZ$,
\begin{align}\label{eq:cspDef}
 \#\{ x\in X : g^d \cdot x = x \} = f(\xi^d)
\end{align}
where $\xi$ is a primitive $n^\thsup$ root of unity.
\end{definition}
Note that it follows immediately from the definition that $\# X = f(1)$. 
	It is almost always tacitly assumed that the group action of $C_n$ on $X$ and the 
	polynomial $f(q)$ should be natural in some sense. 
	The group action could be some form of rotation or cyclic shift of 
	the elements of $X$. The polynomial usually has a closed form and is 
	also typically the generating polynomial for some combinatorial statistic defined on $X$. 

In \cref{table: instances of CSP} we summarize some of the most famous and relevant instances of cyclic sieving. 
For a more comprehensive list, see B. Sagan's article~\cite{Sagan2011}.

\begin{table}[!ht]
	\centering
	\begin{tabular}{llll}
		\toprule
		Set & Group action & Polynomial &  Reference \\ 
		\midrule 
		$k$-subsets of $[n]$ & Nearly free action & $\qbinom{n}{k}_q$ & \cite{ReinerStantonWhite2004}\\
		Words with content $\alpha$ & Cyclic shift & $\qbinom{|\alpha|}{\alpha_1, \dots, \alpha_\ell}_q$ & \cite{ReinerStantonWhite2004}\\
		Non-cross. perf. matchings & Rotation & $\frac{1}{[n+1]_q}\qbinom{2n}{n}_q$ & \cite{PetersenPylyavskyyRhoades2008}\\
		$\SYT(n^m)$ & Promotion &  $f^\lambda(q)$ & \cite{Rhoades2010} \\
		$01$-matrices & Shift rows/columns & See \cref{thm:rhoadesMatrices} & \cite{Rhoades2010b}\\
		\bottomrule
	\end{tabular}
	\caption{
		A few known instances of cyclic sieving.
	}	\label{table: instances of CSP}
\end{table}
One of the main results of this paper, \cref{thm:mainCSP}, is a generalization of the first instance of cyclic sieving
in \cref{table: instances of CSP} and it is also closely related to the last instance in the table. 

\begin{example}[$k$-subset CSP, see \cite{ReinerStantonWhite2004}]\label{ex:CSPonSubsets}
	Let $\binom{[n]}{k}$ be the set of $k$-subsets of $[n]$.
	Suppose that $C_N$ is generated by a permutation $\sigma \in \symS_n$, 
	where the cycles of $\sigma$ consists of $N$-cycles and one or zero singletons. 
	Let $C_N$ act on $[n]$ in the natural way (this is referred to as $C_N$ acting \defin{nearly freely} on $[n]$). 
	Then $\left(\binom{[n]}{k}, C_N, \qbinom{n}{k}_q \right)$ exhibits the cyclic sieving phenomenon.
\end{example}

The situation is even more interesting when different instances of the 
cyclic sieving phenomenon are related in a certain fashion.

\begin{definition}[Lyndon-like CSP, \cite{AlexanderssonLinussonPotka2019}]\label{def:LyndonLike}
Let $\{X_n\}_{n=1}^\infty$ be a family of combinatorial objects
with a cyclic group action $C_n$ acting on $X_n$.
Furthermore, let $\{f_n(q)\}_{n=1}^\infty$ be a sequence of polynomials in $\setN[q]$,
such that for each $n=1,2,\dotsc,$ the triple $(X_n,C_n,f_n(q))$ exhibits the cyclic sieving phenomenon.
We say that the family of triples $\{(X_n,C_n,f_n(q))\}_{n=1}^\infty$ 
is \defin{Lyndon-like} if $f_{n/d}(1) = f_n(e^{\frac{2\pi i}{d}})$ 
for all positive integers $d$, $n$ such that $d|n$.
\end{definition}
Phrased in a different manner, the family is Lyndon-like if and only if 
the number of elements in $X_n$ fixed by $g^d$ is in bijection with $X_d$
where $g$ is an element of order $n$ in $C_n$. 
We note that the notion of Lyndon-like is also 
studied from a different perspective (called $q$-Gauß congruences) in \cite{Gorodetsky2019}.

\begin{example}\label{ex:CSPonwords}
Let $W_{nk}$ be the set of words of length $n$ in the alphabet $[k]$. 
Let $C_n$ act on $W_{nk}$ by cyclic rotation. 
Take $f_n(q)=\sum_{w \in W_{nk}} q^{\maj(w)}$. 
Then $(W_{nk}, C_n, f_n(q))$ exhibits the cyclic sieving phenomenon. 
Furthermore, if we fix $k$, this family of CSP-triples is Lyndon-like.
\end{example}

One can show that the group action on a Lyndon-like family $X_n$ corresponds to 
rotation on some set of words of length $n$. 
When $X_n$ is the set of binary words of length $n$,
the orbits of length $n$ are in bijection with \emph{Lyndon words}, see \oeis{A001037} in \cite{OEIS}.
Each Lyndon-like family of combinatorial objects then has an analogue of Lyndon words.

\subsection{Symmetric functions and plethysm}

We use standard notation (see e.g. \cite{Macdonald1995,StanleyEC2})
for symmetric functions. We have the 
elementary symmetric functions $\elementaryE_\lambda$,
complete homogeneous symmetric functions $\completeH_\lambda$,
the power-sum symmetric functions $\powerSum_\lambda$
and the Schur functions $\schurS_\lambda$.
Recall also the standard involution on symmetric functions $\omega$,
with the defining properties that for $\lambda \vdash n$,
\[
 \omega(\completeH_\lambda) = \elementaryE_\lambda,\qquad 
 \omega(\schurS_\lambda) = \schurS_{\lambda'}, \qquad 
 \omega(\powerSum_\lambda) = (-1)^{n-\length(\lambda) }\powerSum_\lambda.
\]
We shall also require a few identities 
related to \defin{plethysm} --- for a comprehensive background on plethysm and 
the notation used, see J. Haglund's book~\cite{qtCatalanBook}. 
In this paper, we only need the following few properties.
When $f$ is a symmetric function, we let 
the \defin{plethystic substitution} $\powerSum_k[f]$ for $k \in \setN$ be defined as
\begin{align}
 \powerSum_k[f] \coloneqq f(x_1^k,x_2^k,x_3^k,\dotsc).
\end{align}
Note that in particular, $\powerSum_k[\powerSum_m] = \powerSum_{km}$.
It is clear from the definition that for symmetric functions $f$ and $g$,
\[
\powerSum_k[f+g] = \powerSum_k[f] + \powerSum_k[g] \text{ and }
\powerSum_k[f\cdot g] = \powerSum_k[f] \cdot \powerSum_k[g].
\]
\begin{lemma}\label{lem:plethOmega}
 For any homogeneous symmetric function $f$ of degree $n$, we have that
 \[
  \powerSum_k[ \omega f ] = (-1)^{(k+1)n} \omega(\powerSum_{k}[f]).
 \]
\end{lemma}
\begin{proof}
Since plethysm is linear, it suffices to prove the identity for $f = \powerSum_\lambda$,
where $\lambda \vdash n$.
We have that $\powerSum_k[ \omega \powerSum_\lambda ] $ is equal to
$\powerSum_k[ (-1)^{n-\length(\lambda)} \powerSum_\lambda ]$ 
$= (-1)^{n-\length(\lambda)} \powerSum_{k\lambda}$
$= (-1)^{(n-\length(\lambda)) + (kn-\length(\lambda))} \omega(\powerSum_{k\lambda})$
which can be simplified to $(-1)^{(k+1)n} \omega(\powerSum_{k}[\powerSum_\lambda])$.
\end{proof}

\subsection{Hall--Littlewood and non-symmetric Macdonald polynomials}

The family of non-symmetric Macdonald polynomials, $\{\macdonaldE_\alpha(\xvec;q,t)\}_\alpha$
where $\alpha \in \setN^n$ is a basis for $\setC[x_1,\dotsc,x_n]$.
These were introduced by E. Opdam~\cite{Opdam1995,Macdonald1995},
and further developed by I. Chrednik~\cite{Cherednik1995nonsymmetric}.
The first definition of non-symmetric Macdonald polynomials 
is quite cumbersome and indirect. 
J. Haglund, M. Haiman and N. Loehr~\cite{HaglundHaimanLoehr2008} found 
a combinatorial formula for computing $\macdonaldE_\alpha(\xvec;q,t)$,
using the notion of \emph{non-attacking fillings},
thus generalizing F. Knop and S. Sahi's earlier formula for Jack polynomials~\cite{KnopSahi1997}.
In this paper, we shall only study a special case of the non-symmetric Macdonald polynomials,
namely the case when $\lambda$ is a partition and $t=0$.
Here, we use the same notation as P. Alexandersson and M. Sawney~\cite{AlexanderssonSawhney2017,AlexanderssonSawhney2019},
which differs slightly from Haglund et al.~\cite{HaglundHaimanLoehr2008}.
The notation $\macdonaldE_\alpha(\xvec;q,t)$ in this paper is equal to 
$E_{\rev(\alpha)}(\xvec;q,t)$ in theirs where the composition has been reversed.
Since we shall only study the specialization $\macdonaldE_\lambda(\xvec;q,0)$, so we do not 
introduce the non-symmetric Macdonald polynomials in full generality.

Let $\lambda=(\lambda_1, \dots, \lambda_r)$ be a 
Young diagram, $m\in \setN$ with $m\geq r$. Let $F : \lambda \to [m]$ be a filling of $\lambda$. 
Three boxes $a$, $b$ $c$ in $F$ form a \defin{triple} if $a$ is just to 
the left of $b$ and $c$ is somewhere under $b$. 
The entries in a triple form an \defin{inversion-triple} if 
they are ordered increasingly in a counter-clockwise orientation. 
If two entries in the triple are equal, then the entry with the 
largest subscript in \eqref{eq:invTriples} is considered to be the biggest.
\begin{equation}\label{eq:invTriples}
\ytableausetup{boxsize=1.1em}
\begin{ytableau}
a_3 & b_1 \\
\none[\circlearrowleft
]  & \none[\scriptstyle\vdots] \\
\none & c_2 \\
\end{ytableau}
\end{equation}
A filling of shape $\lambda$ is called a \defin{coinversion-free filling}
if every triple is an inversion-triple and the first column is 
strictly decreasing from top to bottom.
The set of such fillings where the entries are in $[m]$ is denoted $\COF(\lambda,m)$.
Note that the conditions imply that every column in a coinversion-free filling 
must have distinct entries.
\begin{remark}
The definition of coinversion-free filling is essentially the same as used 
by P.~Alexandersson and M.~Sawhney~\cite{AlexanderssonSawhney2017} and by J.~Uhlin~\cite{Uhlin2019} 
with the exception that the aforementioned texts also include \emph{basements}. 
However, it is easy to see that these different definitions both yield $\macdonaldE_\lambda(\xvec;q,0)$. 
Arguably, our definition makes the results in this article more natural. 
S.~Assaf~\cite{Assaf2018Kostka} and S.~Assaf, N.~Gonz\'ales~\cite{AssafGonzalez2018} 
study a generalized form of coinversion-free fillings, 
which also allows composition-shaped fillings. 
Therein, they are called \defin{semistandard key tabloids}.
\end{remark}
A \defin{descent}\footnote{Note that this seems non-standard compared to descents in words. 
This terminology is due to the usage of \emph{skyline} diagrams used
when describing the non-symmetric Macdonald polynomials~\cite{HaglundHaimanLoehr2008}. 
We use English notation rather than skyline diagrams.
} of a filling $F$ is a box $(i,j)$ such that $F(i,j-1)<F(i,j)$.
In particular, there are no descents in the first column.
The set of descents is denoted $\Des(F)$. The \defin{leg} of a box $b$ is the number 
of boxes that lie strictly to the right of $b$ in the diagram. 
In other words, if $b=(i,j)$, then $\leg(b)=\lambda_i-j$. 
The \defin{major index} of $F$ is defined as
\[
 \maj(F) = \sum_{b \in \Des(F)} (\leg(b)+1)).
\]
Given a filling $F$ of shape $\lambda$, we let the \defin{weight} of $F$, $\weight(F) = (w_1,\dotsc,w_m)$,
be the vector such that $w_i$ count the number of occurrences of $i$ in $F$.
Furthermore, we let $\xvec^F$ denote the monomial $\prod_{(i,j)\in \lambda} x_{F(i,j)}$,
see \cref{fig:descents}.
\begin{figure}[!ht]
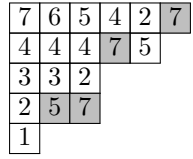

\ytableausetup{boxsize=1.1em}
	\[
	\ytableaushort{76542{*(lightgray)7},444{*(lightgray)7}5,332,2{*(lightgray)5}{*(lightgray)7},1}
	\]
	\caption{A coinversion-free filling with descents marked and weight $(1,3,2,4,3,1,4)$.
	Its major index is $1+2+2+1 = 6$.
	}\label{fig:descents}
\end{figure}

The (specialized) \defin{non-symmetric Macdonald polynomial} $\macdonaldE_\lambda(\xvec;q,0)$
is then defined as 
\begin{equation}\label{eq: combinatorialFormula}
 \macdonaldE_\lambda(x_1,\dotsc,x_m;q,0) = \sum_{F \in \COF(\lambda,m)} q^{\maj(F)} \xvec^F.
\end{equation}
One can verify that this definition agrees with the one 
given in \cite{Alexandersson2015gbMacdonald}.
Despite the name, the specialization $\macdonaldE_\lambda(\xvec;q,0)$
is in fact a symmetric polynomial\footnote{There is an extension of the 
notion of inversion triples to diagrams indexed by weak compositions and 
then one has that $\macdonaldE_\lambda(\xvec;q,0)$ 
is independent of the order of entries in $\lambda$,
see \cite[Prop. 17]{Alexandersson2015gbMacdonald}.}
as we shall see below.

\begin{example}
As can be computed by summing all monomials in \cref{fig:SpecializedFillings}, 
we have that $\macdonaldE_{(2,1)}(x_1,x_2,x_3;q,0)$ is given by
\begin{align*}
	& x_1^2x_2+x_1x_2^2+x_1^2x_3+x_1x_3^2+x_2^2x_3+x_2x_3^2+2x_1x_2x_3+qx_1x_2x_3 \\
	&= (2+q)\monomial_{111}(x_1,x_2,x_3) + \monomial_{21}(x_1,x_2,x_3).
\end{align*}
\begin{figure}[!ht]
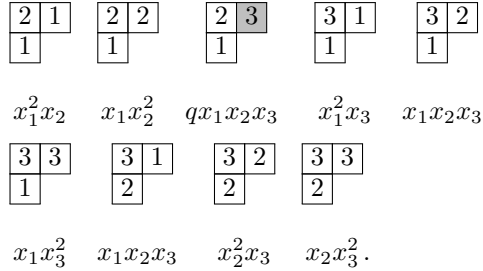

	\begin{align*}
	&\overset{\ytableaushort{21,1,\none}}{ x_1^2x_2 } \quad
	\overset{\ytableaushort{22,1,\none}}{ x_1x_2^2} \quad
	\overset{\ytableaushort{2{*(lightgray)3},1,\none}}{ qx_1x_2x_3 \phantom{\!\!\!\!x^2}} \quad
	\overset{\ytableaushort{31,1,\none}}{ x_1^2x_3} \quad
	\overset{\ytableaushort{32,1,\none}}{ x_1x_2x_3 \phantom{\!\!\!\!x^2}}  \\
	&\overset{\ytableaushort{33,1,\none}}{ x_1x_3^2} \quad
	\overset{\ytableaushort{31,2,\none}}{ x_1x_2x_3 \phantom{\!\!\!\!x^2}} \quad
	\overset{\ytableaushort{32,2,\none}}{ x_2^2x_3} \quad
	\overset{\ytableaushort{33,2,\none}}{ x_2x_3^2}.
	\end{align*}
	\caption{All coinversion-free fillings of shape $\lambda=(2,1)$ in three 
	variables with their respective contribution to \cref{eq: combinatorialFormula}.}
	\label{fig:SpecializedFillings}
\end{figure}
\end{example}

We shall also briefly make use of the \emph{modified Macdonald polynomials} further down.
Let $\lambda/\mu $ be a skew shape and let $F : \lambda/\mu \to [m]$ be a filling with no restrictions.
The notion of \emph{inversion triples} with the cases in \eqref{eq:invTriples} 
is extended to skew shapes, where the box $a$ may now lie outside the diagram.
\begin{equation}\label{eq:skewInversionTriple}
 \begin{ytableau}
 a_3 & b_1 \\
 \none[\circlearrowleft]  & \none[\scriptstyle\vdots] \\
\none & c_2 \\
\end{ytableau}
\qquad 
\qquad 
 \begin{ytableau}
 \infty & b_1 \\
 \none[\circlearrowleft]  & \none[\scriptstyle\vdots] \\
\none & c_2 \\
\end{ytableau}
\end{equation}
 Such boxes outside the diagram $\lambda/\mu$ are considered to have value $\infty$,
 in which case it is required $F(b_1)>F(c_2)$ in order for the triple to be an inversion triple.
The notion of descent is extended to skew shapes, so that $(i,j)$ is a descent of $F$
 if $F(i,j-1)<F(i,j)$, and we let $\maj(F) = \sum_{b \in \Des(F)} (\leg(b)+1)$
 as before. Similarly, the notion of \defin{weight} is extended in the natural way.
The (skew) \defin{modified Macdonald polynomial} $\macdonaldH_{\lambda/\mu}(\xvec;q,t)$
is defined as
\begin{equation}\label{eq:macdonaldHDef}
 \macdonaldH_{\lambda/\mu}(x_1,\dotsc,x_m;q,t) = \sum_{F : \lambda/\mu \to [m]} q^{\maj(F)} t^{\inv(F)} \xvec^F.
\end{equation}
This stabilizes to a symmetric function as $m \to \infty$.
Note that $[t^*]\macdonaldH_{\lambda}(\xvec;q,t) = \macdonaldE_\lambda(\xvec;q,0)$,
that is, the coefficient of the highest power of $t$ that 
appears is given by a specialized Macdonald polynomial.

\begin{proposition}[See \cite{AlexanderssonSawhney2017,AlexanderssonSawhney2019}]\label{prop:columnSets}
Let $\lambda$ be a partition with $\ell$ parts
and let $S_1,S_2,\dotsc,S_\ell$ be subsets of $[m]$ such that $|S_j| = \lambda_j$.
Then there is a unique coinversion-free filling $F$ of shape $\lambda'$ 
such that the entries in column $j$ of $F$ are given by $S_j$.
\end{proposition}
\cref{prop:columnSets} implies that $\macdonaldE_{\lambda'}(\xvec;1;0) = \elementaryE_\lambda(\xvec)$.
In fact, this identity generalizes to the non-symmetric setting, see~\cite{AlexanderssonSawhney2019}.

We shall also make use of the \defin{transformed Hall--Littlewood polynomials},
$\hallLittlewoodT_{\mu}(\xvec;q)$.
There are many different ways to define these, for example via the 
\defin{Kostka--Foulkes polynomials} $K_{\lambda\mu}(q)$:
\begin{align}\label{eq:kostkaFoulkes}
 \hallLittlewoodT_{\mu}(\xvec;q) = \sum_{\lambda} K_{\lambda\mu}(q) \schurS_\lambda(\xvec).
\end{align}
We refer to \cite{Macdonald1995,DesarmenienLeclercThibon1994,TudoseZabrocki2003} for more background
and properties. 
For completeness, we provide a combinatorial method 
for computing $K_{\lambda\mu}(q)$ in \cref{sec:kostkaFoulkes}.
The transformed Hall--Littlewood polynomials are closely related to 
our specialization of the non-symmetric Macdonald polynomials.
\begin{proposition}[{See \cite[Thm. 14]{AlexanderssonSawhney2017} or \cite{qtCatalanBook}}]\label{prop:EasHL}
For any partition $\lambda$
 \[
  \macdonaldE_{\lambda}(\xvec;q;0) = \omega \hallLittlewoodT_{\lambda'}(\xvec;q).
 \]
\end{proposition}

\subsection{Evaluations at roots of unity}
 
 Our first goal is to evaluate $\macdonaldE_{\lambda}(\xvec;q;0)$ when $q$
 is a root of unity. We need a few intermediate results.
 
\begin{proposition}[{See \cite[Thm. 2.1]{LascouxLeclercThibon1994}}]\label{prop:hlUnityRoot}
Let $\lambda$ be a partition of $n$.
Let $d$ be a positive integer and write $m_j(\lambda) = d m'_j+r_j$ with $0 \leq r_j < d$.
Furthermore, let $\xi$ be a primitive $d^\thsup$ root of unity. Then
\begin{align}
 \hallLittlewoodT_{\lambda}(\xvec;\xi) = 
 \hallLittlewoodT_{\tilde{\lambda}}(\xvec;\xi) \prod_{j=1}^n
 \left( \hallLittlewoodT_{(j^d)}(\xvec;\xi) \right)^{m'_j}
\end{align}
where $\tilde{\lambda}$ is the partition $(1^{r_1},2^{r_2},\dotsc,n^{r_n})$.
\end{proposition}

\begin{proposition}[{See \cite[Thm. 2.2]{LascouxLeclercThibon1994}}]\label{prop:hlRectUnityRoot}
 Let $\xi$ be a primitive $n^\thsup$ root of unity. Then
 \begin{align}
  \hallLittlewoodT_{k^n}(\xvec;\xi) = (-1)^{k(n-1)}\powerSum_n[\completeH_k(\xvec)].
 \end{align}
\end{proposition}

\begin{theorem}\label{thm:macdonaldEAtUnityRoots}
Let $\mu$ be an integer partition and $n \in \setN$. 
Furthermore, let $\xi$ be a primitive $n^\thsup$ root of unity
and $d$ a divisor of $n$. Then 
\begin{align}\label{eq:macdonaldEAtUnityRoots}
 \macdonaldE_{n\mu}(1^m; \xi^d,0) = \prod_{j\geq 1} \binom{m}{j}^{d\cdot m_j(\mu')}.
\end{align}
\end{theorem}
\begin{proof}
We have that $\xi^d$ is a primitive $\left(\frac{n}{d}\right)^\thsup$ root of unity.
By \cref{prop:hlUnityRoot},
\begin{equation*}
 \hallLittlewoodT_{(n\mu)'}(\xvec; \xi^d) = 
 \prod_{j\geq 1}
 \left(\hallLittlewoodT_{j^{n/d}}(\xvec; \xi^d)\right)^{d \cdot  m_j(\mu')}.
\end{equation*}
Using \cref{prop:hlRectUnityRoot}, we then have
\begin{equation}\label{eq:hlRectPleth}
 \hallLittlewoodT_{(n\mu)'}(\xvec; \xi^d)  = 
 \prod_{j\geq 1}
 \left( (-1)^{j(n/d-1)}\powerSum_{n/d}[\completeH_j(\xvec)] \right)^{d \cdot  m_j(\mu')}.
\end{equation}
Applying the $\omega$-involution on both sides of \eqref{eq:hlRectPleth} and using 
\cref{prop:EasHL} and \cref{lem:plethOmega} gives
\begin{align*}
 \macdonaldE_{n\mu}(\xvec; \xi^d,0) 
 =  \prod_{j\geq 1}
 \left( \powerSum_{n/d}[\elementaryE_j(\xvec)] \right)^{d \cdot  m_j(\mu')}
 =  \prod_{j\geq 1}
 \left( \elementaryE_j(x_1^{\frac{n}{d}},x_2^{\frac{n}{d}},\dotsc)  \right)^{d \cdot  m_j(\mu')}.
\end{align*}
Finally, we set $(x_1,\dotsc,x_m)=(1,q,q^2,\dotsc,q^{m-1})$ and remaining variables are set to zero.
Formulas for the principal specialization of elementary symmetric functions, see \cite[p. 303]{StanleyEC2},
give
\begin{equation}\label{eq:macdonaldPrincSpec}
 \macdonaldE_{n\mu}(1,q,q^2,\dotsc,q^{m-1}; \xi^d,0) =  
 \prod_{j\geq 1} \left( q^{\frac{n}{d}\binom{j}{2}} \qbinom{m}{j}_{q^\frac{n}{d}} \right)^{d \cdot m_j(\mu')}.
\end{equation}
In particular, with $q=1$ we obtain \eqref{eq:macdonaldEAtUnityRoots}.
\end{proof}

\begin{example}
 Let $\mu = 885322$, $n=12$ and $d=4$. 
 Then $\mu'=66433222$ and
 \[
 \macdonaldE_{12\mu}(1^m; e^{\frac{2\pi i 4}{12}},0) = \left(\binom{m}{2}^{3}\binom{m}{3}^{2}
 \binom{m}{4}\binom{m}{6}^2 \right)^4.
 \]
\end{example}

\section{Cyclic sieving on coinversion-free fillings}\label{sec:csp1}

Recall that $\COF(n \lambda,m)$ denotes the set of coinversion-free fillings 
of shape $n \lambda$ and entries in $[m]$.
Let $\phi$ act on $\COF(n \lambda,m)$ by cyclically shifting the 
first $n$ columns one step to the right,
the next $n$ columns one step to the right, and so on.
Finally, the elements in each column are rearranged so that the result is 
again a coinversion-free filling in $\COF(n \lambda, m)$.
This action is well-defined according to \cref{prop:columnSets}.
Clearly, $\phi$ generates a cyclic group of order $n$ acting on $\COF(n\lambda,m)$,
see \cref{fig:phiAction} for an example.
\begin{figure}[!ht]
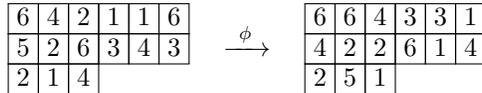

	\begin{equation*}
	\ytableaushort{642116, 526343, 214} \quad \xrightarrow{ \ \phi \ } \quad \ytableaushort{664331, 422614, 251}
	\end{equation*}
	\caption{The action of $\phi$ on a filling of shape $n\lambda = 3(2,2,1)$.}   
	\label{fig:phiAction}  
\end{figure}
We are now ready to prove one of the main results of the paper.
\begin{theorem}\label{thm:mainCSP}
For every integer partition $\lambda$ and $n\geq 1$, the triple 
\[
 \left(\COF(n \lambda,m), \langle \phi \rangle, \macdonaldE_{n \lambda}(1^m,q,0) \right)
\]
exhibits the cyclic sieving phenomenon. 
Moreover, the family 
\[
\{ \left(\COF(n \lambda,m), \langle \phi \rangle, \macdonaldE_{n \lambda}(1^m,q,0) \right) \}_{n=1}^\infty
\]
is a Lyndon-like family.
\end{theorem}
\begin{proof}
We first need to compute the number of elements fixed under $\phi^d$ whenever $d|n$.
Since coinversion-free fillings are uniquely determined by their column sets,
a coinversion-free filling in $\COF(n \lambda,m)$ fixed under $\phi^d$
is uniquely determined by the first $d$ columns in each consecutive block of $n$
columns.
Hence,
\begin{equation}
 \#\{ F \in \COF(n \lambda ,m) : \phi^d \cdot F = F\} = \#\COF\left(d\lambda ,m\right) = 
 \prod_{j\geq 1}\binom{m}{k}^{d \cdot m_j(\lambda')}.
\end{equation}
By using \cref{thm:macdonaldEAtUnityRoots}, the statements in
the theorem now follows.
\end{proof}

There is another natural group action on coinversion-free fillings. 
Pick $\sigma \in \symS_m$ and let $\sigma : \COF(\lambda, m) \to \COF(\lambda, m)$
act by letting $F' = \sigma(F)$ be given by $F'(i,j)=\sigma F(i,j)$ for all $(i,j)\in \lambda$,
followed by rearranging the elements in each column of $F'$ 
to obtain a new coinversion-free filling.
By \cref{prop:columnSets}, this action is well defined. 
Let $C_M = \langle \sigma \rangle$. 
Recall the notion of $C_M$ acting nearly freely 
on $[m]$ from \cref{ex:CSPonSubsets}. This induces a $C_M$-action on $\COF(\lambda, m)$.

\begin{theorem}
	Let $n$ and $m$ be positive integers and let $\lambda$ be a partition. If $C_n$ acts nearly freely on $[m]$, then  
	\[
	\left(\COF(\lambda,m), C_n, \macdonaldE_\lambda(1,q,q^2,\dots,q^{m-1};1,0) \right)
	\]
	exhibits the cyclic sieving phenomenon.
\end{theorem}
\begin{proof}
The proof of the case when $\lambda$ is rectangular can be found in \cite{Uhlin2019},
and extends to the general case without extra effort. 
Suppose $C_n$ is generated by $\sigma$. 
Write $\lambda'=(\lambda_1',\dots, \lambda_c')$ and consider the set of $c$-tuples $\mathbf{s}=(s_1, \dots, s_c)$, so that $s_i \subseteq [m]$ and $\#s_i=\lambda_i'$. 
We let $C_n$ act on such $c$-tuples by letting 
$\sigma \mathbf{s}=(\sigma s_1, \dots, \sigma s_c)$. 
\cref{prop:columnSets} implies that we can identify a coinversion-free filling with its columns sets. 
Hence, there is a bijection from $\COF(\lambda,m)$ to the set of $c$-tuples on the above form 
and it is clear that this bijection is equivariant with respect to $\sigma$, 
so it suffices to show that the set of $c$-tuples exhibits the cyclic sieving phenomenon. 
	
But the set of $c$-tuples is a direct product of sets which we know 
exhbit the cyclic sieving phenomenon --- this is just \cref{ex:CSPonSubsets}. 
Furthermore, the product of the CSP-polynomials in the example agree with \eqref{eq:macdonaldPrincSpec} (when $\xi=1$). 
It is straightforward to show that that CSP is preserved under
taking direct products (see \cite[Lem. 4.13 $(i)$]{Uhlin2019}) so we are done.
\end{proof}
\begin{example}
	Note that the above theorem does in general not hold if $C_M$ does not act nearly freely on $[m]$. 
	Suppose that $n=k=1$ and $C_4$ is generated by the 
	permutation $\langle (1234)\rangle$. 
	Then, the CSP-polynomial $\macdonaldE_{n^k}(1,q,q^2,\dots,q^{m-1};1,0)=1+q+q^2+q^3+q^4+q^5$, 
	which evaluated at a primitive $4^\thsup$ of unity $\xi=i$ yields $1+i$. 
\end{example}

\begin{figure}[!ht]
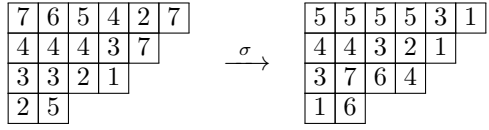

	\begin{equation*}
	\ytableaushort{765427, 44437, 3321, 25} \quad \xrightarrow{ \ \sigma \ } \quad \ytableaushort{555531, 44321, 3764, 16} 
	\end{equation*}
	\caption{An example of the action of $\sigma=(1234567)$.}
\end{figure}

\section{Refined CSP on $\COF(n\lambda)$}\label{sec:cspRefined}

Let $[\monomial_\nu]f$ denote the coefficient of $\monomial_\nu$
in the symmetric function $f$.
It follows from \cref{prop:EasHL} and \eqref{eq:kostkaFoulkes} that
\begin{equation}\label{eq:macdonaldECoefficient}
  [\monomial_\nu] \macdonaldE_{\lambda}(\xvec;q;0) =\sum_{\mu} K_{\mu\nu}(1) K_{\mu'\lambda'}(q).
\end{equation}

\begin{theorem}[{See \cite[Thm. 1.4]{Rhoades2010b}}]\label{thm:rhoadesMatrices}
Let $\mu$ and $\nu$ be compositions of $n$, with cyclic symmetries $a$ and $b$, respectively.
Let $X(\mu,\nu)$ be the set of $\length(\lambda) \times \length(\nu)$
binary matrices with row content $\mu$ and column-content $\nu$.
Then the product $\setZ_{\length(\mu)/a} \times \setZ_{\length(\nu)/b}$ act on $X$
by $a$-fold row-rotation and $b$-fold column-rotation, respectively.
Then 
\[
 \left(X(\mu,\nu), \setZ_{\length(\mu)/a} \times \setZ_{\length(\nu)/b},
 \delta_n(q,t) \sum_{\lambda \vdash n} K_{\lambda\mu}(q) K_{\lambda' \nu}(t)
 \right)
\]
exhibits the bi-cyclic sieving phenomenon.
Here, $\delta_n(q,t)$ is a messy polynomial taking on values  $\pm 1$
at relevant roots of unity. Furthermore, one can check that $\delta_n(q,t) \equiv 1$ in the case $q=1$.
\end{theorem}

By only considering the action on the 
columns in \cref{thm:rhoadesMatrices}, we get the following corollary.
\begin{corollary}
Let $n \in \setN$, $\lambda \vdash k$ where $\ell = \length(\lambda)$, 
$\nu \vDash nk$ with $m$ parts and let $X(\nu, n \lambda)$ 
be the set of $m \times n \lambda_1$ binary matrices 
with row-content $\nu$ and column-content given by the conjugate of $n\lambda$.
Let $\setZ_n$ act on $X(\nu, n \lambda)$ by cyclic rotation of each block of $n$ consecutive columns.
Then 
\[
 \left(X(\nu, n \lambda),\setZ_n, \sum_{\mu \vdash nk} K_{\mu,\nu}(1) K_{\mu',(n\lambda)'}(q) \right)
\]
is a CSP-triple.
\end{corollary}
\begin{proof}
Note that there is an easy correspondence between $n$-fold rotation of columns 
and rotation of each block of $n$ consecutive columns.
\end{proof}

Recall the definition of $\phi$ in \cref{sec:csp1},
which cyclically shifts each block of $n$ consecutive columns.
\begin{theorem}\label{thm:refinedMacdonaldCSP}
Let $n \in \setN$, $\lambda \vdash k$ and 
let $\nu$ be a weak composition of $nk$ with $m$ parts.
Then 
\[
 \left(\COF(n \lambda, \nu ), \langle \phi \rangle, [\monomial_\nu]\macdonaldE_{n \lambda}(\xvec,q,0) \right)
\]
exhibits the cyclic sieving phenomenon.
\end{theorem}
\begin{proof}
Let $\varphi$ denote the one-step cyclic rotation of each block of $n$ consecutive columns in a matrix.
By \eqref{eq:macdonaldECoefficient} it suffices to show that
there is a bijection $A : X(\nu, n \lambda) \to \COF(n \lambda,\nu)$
such that $\phi\circ A(M) = A\circ \varphi(M)$.
We let $A(M) = F$ if and only if
\[
 M(i,j)=1 \qquad \iff \qquad \text{column $j$ in $F$ contains $i$}.
\]
It is straightforward to verify that $\phi\circ A(M) = A\circ \varphi(M)$,
so fixed-points under $\varphi^d$ are mapped to fixed-points under
$\phi^d$ for all $d \in \setZ$. This proves the theorem.
\end{proof}
See \cref{fig:macdonaldE-orbits} for an illustration of \cref{thm:refinedMacdonaldCSP}.

\ytableausetup{boxsize=0.8em}
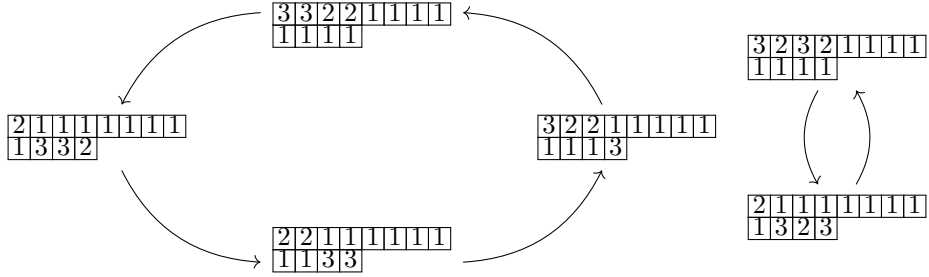
\begin{figure}[!ht]
	\begin{tikzcd}
		&  \ytableaushort{33221111,1111} \arrow[ld, bend right] \\
		\ytableaushort{21111111,1332} \arrow[rd, bend right] & &  \ytableaushort{32211111,1113} \arrow[lu, bend right]\\
		&  \ytableaushort{22111111,1133} \arrow[ru, bend right]
	\end{tikzcd}
	\begin{tikzcd}
		\ytableaushort{32321111,1111} \arrow[dd, bend right] \\  \\
		\ytableaushort{21111111,1323} \arrow[uu, bend right]
	\end{tikzcd}	
	
	\caption{The orbits of $\COF(84, (8,2,2))$ under $\phi$. 
	The CSP-polynomial is $f(q)=[\monomial_{(8,2,2)}]\macdonaldE_{84}(\xvec,q,0)=1+q+q^2+q^3+q^4+q^6$.
	One can easily check that for a primitive $4^\thsup$ root of unity $\xi$,
	$f(\xi)=0$, $f(\xi^2)=2$ and $f(\xi^4)=6$ and thus agreeing with the definition of CSP in \eqref{eq:cspDef}.
	}\label{fig:macdonaldE-orbits}
\end{figure}

\section{Skew specialized Macdonald polynomials}\label{sec:skew}

There is a natural generalization of $\macdonaldE_{\lambda}(\xvec;q,0)$ to skew diagrams.
In this section, we shall see that $\macdonaldE_{\lambda/\mu}(\xvec;q,0)$ is symmetric and Schur positive.
Interestingly, the coefficients in the Schur expansion are not 
related to skew Kostka--Foulkes polynomials which at first glance is a natural guess.

A \defin{skew specialized Macdonald filling of shape $\lambda/\mu$}
is a filling $F$ of the skew shape $\lambda/\mu$ such that
each column of $F$ contains distinct entries,
the first column is strictly decreasing, and
every triple in $F$ is an inversion-triple as in \eqref{eq:skewInversionTriple}.
We let $\COF(\lambda/\mu)$ denote the set of all such fillings. 
It is not difficult to see that \cref{prop:columnSets} can be generalized 
to the skew-case as well. 
In other words a skew specialized Macdonald filling is 
completely determined by the shape of the diagram and the ordered tuple of column sets. 
\begin{definition}
 Let $\lambda/\mu$ be a skew shape, and define the \defin{skew} (specialized) \defin{non-symmetric Macdonald polynomial} as
 \[
\macdonaldE_{\lambda/\mu}(\xvec;q,0) = \sum_{F \in \COF(\lambda/\mu)}\xvec^F q^{\maj(F)}.
 \]
\end{definition}
One can quite easily see that these polynomials generalize the skew Schur functions:
\begin{equation}\label{eq:skewMacdonaldEAtqZero}
\macdonaldE_{\lambda/\mu}(\xvec;0,0) = \schurS_{\lambda/\mu}(\xvec).
\end{equation}
As in the non-skew case, the functions $\macdonaldE_{\lambda/\mu}(\xvec;q,0)$ are actually symmetric,
and we are justified to work in any number of variables.
It is clear from the definition that 
\begin{equation}\label{eq:macdonaldEAsSkewMacdonaldH}
\macdonaldE_{\lambda/\mu}(\xvec;q,0) \coloneqq [t^*]\macdonaldH_{\lambda/\mu}(\xvec;q,t).
\end{equation}
The fact that these are symmetric follows from \cref{thm:skewEInSchurExpansion} below.
Symmetry was proved earlier in the non-skew case \cite{Uhlin2019} by using a 
variant of the \emph{Lascoux--Sch{\"{u}}tzenberger involutions},
see \cref{def:lsInv} below.

Given a filling $F \in \COF(\lambda/\mu)$ and some large integer $M$, 
we define the \defin{extended filling} $\hat{F}$
as the filling of shape $\lambda$ obtained from $F$ as
\begin{equation}\label{eq:extendedFilling}
\hat{F}(i,j) = 
\begin{cases}
 M- i \text{ if } (i,j)\in \mu \\
 F(i,j) \text{ otherwise}.
\end{cases}
\end{equation}
Note that $\hat{F}$ is a specialized Macdonald filling and that $\maj(F) = \maj(\hat{F})$. 
We shall make use of this definition in the next subsection.
\begin{figure}[ht!]
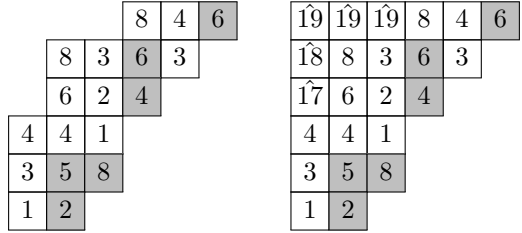

\[
\ytableausetup{boxsize=1.4em}
\ytableaushort{\none \none \none 84{*(lightgray)6},\none 83{*(lightgray)6}3,  \none 62{*(lightgray)4}, 
441,3{*(lightgray)5}{*(lightgray)8},1{*(lightgray)2}}
\qquad 
\ytableaushort{{\hat{19}}{\hat{19}}{\hat{19}} 84{*(lightgray)6},{\hat{18}} 83{*(lightgray)6}3, {\hat{17}} 62{*(lightgray)4}, 
441,3{*(lightgray)5}{*(lightgray)8},1{*(lightgray)2}}
\]
\caption{Left: A skew specialized Macdonald filling of shape $\lambda/\mu$ with $\lambda
=(6,5,4,3,3,2)$ and $\mu=(3,1,1)$ with descents marked. 
The weight of the filling is $(2,2,3,4,1,3,0,3)$.
Right: $\hat{F}$ with $M=20$.
}
\end{figure}

\subsection{Charge, RSK and Schur expansion}\label{ssec:schurExp}

We assume that the reader is familiar with the (row-insertion) Robinson--Schenstedt--Knuth correspondence, (RSK),
described briefly in \cref{sec:burgeRSK}.
See the appendix, \cref{sec:kostkaFoulkes}, for the definition of 
charge on words and semistandard Young tableaux.
\begin{definition}
Let $\mu$ be a partition and let $w$ be a word with content $\beta$ (which can be a weak composition)
such that $\mu+\beta$ is a partition.
The \defin{postfix charge} $\charge_{\mu}(w)$ is defined via the usual charge statistic as 
\[
 \charge_{\mu}(w) \coloneqq \charge\left( w \; \cdot \; \ell^{\mu_{\ell}} \dotsm 2^{\mu_{2}} 1^{\mu_{1}} \right).
\]
That is, we concatenate a postfix to $w$ with content $\mu$, where the letters appear in decreasing order.
For example, $\charge_{21}(12231233) = \charge(12231233\cdot 211) = 2$.
\end{definition}

Recall the definition of \defin{elementary Knuth transforms},
stating that $yzx \sim_K yxz $ whenever $x<y\leq z$
and $xzy \sim_K zxy$ whenever $x \leq y < z$.
Two words are \defin{Knuth-equivalent}, if one can obtain one from the other via a sequence
of elementary Knuth transforms.
If $w$ has partition-content, then its equivalence class contains a unique word 
which is the reading word of a semistandard Young tableau.
Moreover, two words of partition content, $w$ and $w'$ are Knuth-equivalent if 
and only if they insert to the same semistandard Young tableau under RSK.

\begin{lemma}\label{lem:postfixCharge}
Let $\mu$ be a partition and suppose $u$ and $v$ are Knuth-equivalent. 
Then 
 $u \cdot \; \ell^{\mu_{\ell}} \dotsm 2^{\mu_{2}} 1^{\mu_{1}}$
 and $v \cdot \; \ell^{\mu_{\ell}} \dotsm 2^{\mu_{2}} 1^{\mu_{1}}$ are 
 Knuth-equivalent and $\charge_\mu(u)=\charge_\mu(v)$.
\end{lemma}
\begin{proof}
The first statement follows easily from the definition of Knuth-equivalence.
Furthermore, if two words are Knuth-equivalent, they have the same charge,
\cite[Cor. 2.4.38]{Butler1994}.
\end{proof}

Recall the notion of a Burge word from \Cref{def:burge word}.
\begin{definition}\label{def:biword}
For each skew specialized Macdonald filling $F$ we associate a \defin{Burge word} $W = W(F)$ as follows. For each entry  $e=F(i,j)$, let $\left(\begin{smallmatrix} j \\ e \end{smallmatrix}\right)$ be a biletter in $W$. Take $W$ to be the (unique) Burge word with all such biletters.
\end{definition}

The non-skew case of \cref{def:biword} was first given in \cite{AlexanderssonSawhney2017}.
Recall \cref{prop:columnSets} and note that we can easily recover the column sets of $F$ from $W$.
The map $W$ from fillings in $\COF(\lambda/\mu,m)$ to biwords
where the top row is a weakly increasing sequence with $j$ entries equal to $\lambda'_j-\mu'_j$,
and bottom row being elements in $[m]$, strictly decreasing on each block of identical elements in the top row,
is therefore a bijection.
Note that $W$ is a bijection when we fix some shape $\lambda/\mu$.
However, two skew specialized Macdonald fillings of \emph{different shapes} may give the same biword.
For example,
\[
\ytableaushort{32,11} \qquad \text{and} \qquad \ytableaushort{\none 2,31,1} 
\]
both have the same biword $\left(\begin{smallmatrix} 1 & 1 & 2 & 2 \\ 3 & 1 & 2 & 1 \end{smallmatrix}\right)$.

%
Note that the content of $F$ is equal to the content of $P$ and is also given by the bottom row of $W$ while the content of $Q$ is given by the top row of $W$.

\begin{figure}[!ht]
\ytableausetup{boxsize=0.9em}
\begin{align*}
\ytableaushort{\none \none 213, 331,
	22,14} \xlongrightarrow{W}
\setlength\arraycolsep{2pt}
\begin{pmatrix}
	1 & 1 & 1 & 2 & 2 & 2 & 3 & 3 & 3 & 4\\
	4 & 3 & 1 & 3 & 2 & 1 & 5 & 2 & 1 & 2
\end{pmatrix} 
\xlongrightarrow{RSK}
\left(
\ytableaushort{1112,225,33,4}\; ,\quad
\ytableaushort{1234,123,13,2}
\right)
\end{align*}
\caption{
A skew specialized Macdonald filling $F$, the corresponding biword,
and the result $(P,Q) = (\ins(F),\rec(F))$ under RSK.}
\label{fig:biword correspondence}
\end{figure}

\begin{proposition}\label{prop:majAspostfixCharge}
Suppose $\lambda/\mu$  is a skew shape and $F \in \COF(\lambda/\mu)$.
Then $\maj(F) = \charge_{\mu'}(\cw(F))$.
\end{proposition}
\begin{proof}
A short proof in the case $\mu=\emptyset$ was given in \cite{AlexanderssonSawhney2019}.
Let $F \in \COF(\lambda/\mu)$ and let $\hat{F}$ be the extended filling of $F$ as in \eqref{eq:extendedFilling}. 
Recall that $\maj(F) = \maj(\hat{F})$. It follows from the definition of the biword that
$
 \cw(\hat{F}) = \cw(F) \cdot \ell^{\mu'_{\ell}} \dotsm 2^{\mu'_{2}} 1^{\mu'_{1}}.
$
The case $\mu=\emptyset$ and \cref{lem:postfixCharge} imply that
$\maj(F) = \maj(\hat{F}) = \charge( \cw(\hat{F}) ) = \charge_{\mu'}(\cw(F))$.
\end{proof}
The Robinson--Schenstedt--Knuth correspondence has the essential property that if the word $w$ inserts to $P$ under RSK, 
then $\charge(w) = \charge(P)$. 
From this property it follows that RSK provides a bijection
\begin{equation}\label{eq:RSKBij}
 \COF(\lambda/\mu) \xleftrightarrow{RSK} \bigcup_{\nu \vdash |\lambda/\mu|} 
 \SSYT(\nu,\alpha) \times \SSYT(\nu'), \quad \alpha_i \coloneqq \lambda'_i - \mu'_i, 
\end{equation}
such that if $F \xleftrightarrow{RSK} (P,Q)$ then 
$\maj(F) = \charge_{\mu'}(\cw(F)) = \charge_{\mu'}(P)$.

We now have the setup needed to prove the following theorem.
\begin{theorem}[The Schur expansion of skew specialized Macdonald polynomials]\label{thm:skewEInSchurExpansion}
Let $\lambda/\mu$ be a skew shape and let $\alpha$ be the 
weak composition given by $\alpha_i \coloneqq \lambda'_i-\mu'_i$.
Then
\begin{equation}\label{eq:skewEInSchurExpansion}
 \macdonaldE_{\lambda/\mu}(\xvec;q,0) \coloneqq 
 \sum_{\nu \vdash |\lambda/\mu|}  \schurS_{\nu'}(\xvec) \sum_{T \in \SSYT(\nu,\alpha)}  q^{\charge_{\mu'}(T)}. 
\end{equation}
\end{theorem}
\begin{proof}
By definition, $\macdonaldE_{\lambda/\mu}(\xvec;q,0)$ is equal to $\sum_{F \in \COF(\lambda/\mu)} \xvec^F q^{\maj(F)}$,
which is equal to $\sum_{F \in \COF(\lambda/\mu)} \xvec^F q^{\charge_{\mu'}(\cw(F))}$
by using \cref{prop:majAspostfixCharge}.
Applying the RSK bijection in \eqref{eq:RSKBij}, we then have that 
\begin{align*}
 \macdonaldE_{\lambda/\mu}(\xvec;q,0) &=  
 \sum_{\nu}  \big( \sum_{P \in \SSYT(\nu',\alpha)}  q^{\charge_{\mu'}(P)} \big) \; \schurS_{\nu'}(\xvec),
\end{align*}
which is exactly the statement in \eqref{eq:skewEInSchurExpansion}.
\end{proof}


\begin{table}[!ht]
\[
 \begin{array}{ll}
  \toprule
  \text{Shape $\lambda/\mu$\hspace{3cm}} & \text{Schur expansion of } \macdonaldE_{\lambda/\mu}(\xvec;q,0) \\
  \midrule
  1		&	\schurS_1 \\
  \midrule
  2		&	\schurS_{2} + q\schurS_{11} \\
  11	&	\schurS_{11} \\
  21/1	&	\schurS_{2} + \schurS_{11} \\
  \midrule
  3		&	\schurS_3 + (q+q^2)\schurS_{21} + q^3 \schurS_{111} \\
  21	&	\schurS_{21} + q\schurS_{111} \\
  111	&	\schurS_{111} \\
  22/1	&	\schurS_{21} + q\schurS_{111} \\
  31/1	&	\schurS_3 + (1+q)\schurS_{21} + q \schurS_{111} \\
  211/1	&	\schurS_{21} + \schurS_{111} \\
  321/21&	\schurS_3 + 2\schurS_{21} + \schurS_{111} \\
  \bottomrule
 \end{array}
\]
\caption{Here are the Schur expansions of 
$\macdonaldE_{\lambda/\mu}(\xvec;q,0)$ in the cases $|\lambda|-|\mu| \leq 3$.
}\label{tab:skewESchurExp}
\end{table}

Recall that the Littlewood--Richardson coefficients $c_{\mu\nu}^{\lambda}$ are defined via the relation
$\schurS_\mu \schurS_\nu = \sum_{\lambda} c_{\mu\nu}^{\lambda} \schurS_\lambda$,
and that the skew Schur functions expand as $\schurS_{\lambda/\mu} = \sum_{\nu} c_{\mu\nu}^{\lambda} \schurS_\nu$.
\begin{corollary}[A Littlewood--Richardson rule]
Let $\lambda/\mu$ be a skew shape and let $\alpha$ be the 
weak composition $\alpha_i \coloneqq	 \lambda'_i-\mu'_i$.
Let
\begin{equation}
K_{\lambda/\mu}^{\nu}(q) \coloneqq \sum_{T \in \SSYT(\nu,\alpha)}  q^{\charge_{\mu'}(T)}.
\end{equation}
Then 
\begin{enumerate}[(a)]
 \item 
 $K_{\lambda/\mu}^{\nu}(0) = c_{\mu{\nu'}}^{\lambda}$, a Littlewood--Richardson coefficient,
 \item 
 $K_{\lambda/\mu}^{\nu}(1) = K_{\nu\alpha}$, a Kostka coefficient,
 \item The product
 \[
  \macdonaldE_\lambda(\xvec;q,0)
  \macdonaldE_\mu(\xvec;q,0)
  =
  \sum_{\nu} K_{\kappa/c^r}^{\nu}(q)   \schurS_{\nu'}(\xvec)
 \]
where $c=\mu_1$, $r=\length(\lambda)$, and 
$\kappa = (c+\lambda_1,\dotsc,c+\lambda_{\length(\lambda)},\mu_1,\mu_2,\dotsc,\mu_{\length(\mu)})$.
\end{enumerate}
\end{corollary}
\begin{proof}
The first and second identity follows 
from \cref{thm:skewEInSchurExpansion} and \eqref{eq:skewMacdonaldEAtqZero}.
The third identity follows from the observation that the product in the left hand 
side can be realized as a single skew specialized Macdonald polynomial,
$\macdonaldE_{(\lambda+c,\mu)/c^r }(\xvec;q,0)$, see the diagram in \eqref{eq:prodAsSkew}.
\begin{equation}\label{eq:prodAsSkew}
\begin{ytableau}
*(lightgray) & *(lightgray) & *(lightgray) & *(lightgray) &  & & & \\
*(lightgray) & *(lightgray) & *(lightgray) & *(lightgray) &  & & \none & \none & \none[\lambda]\\
*(lightgray) & *(lightgray) & *(lightgray) & *(lightgray) & & \\
 & & & \\
 & &  \\
 &  \none & \none & \none[\mu]
\end{ytableau}
\end{equation}
\end{proof}
The coefficients $K_{\lambda/\mu}^{\nu}(q)$ might be related 
to the \emph{parabolic Kostka polynomials}, whose constant 
terms are also Littlewood--Richardson coefficients, 
see \cite{ShimozonoWeyman2000,KirillovSchillingShimozono2001} for details.

We end this subsection with proving an additonal property of postfix-charge.
Recall that the charge statistic is \defin{Mahonian}, meaning that
$\sum_{\sigma \in \symS_n} q^{\charge(\sigma)} = [n]_q!$.
There is a natural generalization of this identity for $\charge_{\mu}(\cdot)$.
\begin{proposition}
Let $\mu \vdash m$ and $n\geq 0$. Then
\[
 \sum_{\sigma \in \symS_n} q^{\charge_{\mu}(\sigma)} = n! \prod_{i \geq 1}
 \frac{[\lambda_i-\mu_i]_q!}{(\lambda_i-\mu_i)!}
\]
where $\lambda' = (\mu'_1+1,\mu'_2+1,\dotsc,\mu'_n+1,\mu_{n+1},\dotsc)$.
\end{proposition}
\begin{proof}
First note that the shape $\lambda/\mu$ has exactly one box in each of the first $n$ columns.
For example, $\mu=53111$ and $n=8$ gives the following skew shape $\lambda/\mu$:
\[
\begin{ytableau}
*(lightgray)&*(lightgray)&*(lightgray)&*(lightgray)&*(lightgray)& \; & \; & \; \\
*(lightgray)&*(lightgray)&*(lightgray)& \; & \; \\
*(lightgray)& \; & \; \\
*(lightgray) \\
 \; 
\end{ytableau}
\]
We use \cref{prop:majAspostfixCharge} together with the fact that 
every permutation appear as charge word appears exactly one when summing over
fillings with weight $1^n$. Hence,
 \begin{equation*}
  [x_1x_2\dotsm x_{n}] \macdonaldE_{\lambda/\mu}(\xvec;q) =  \sum_{\sigma \in \symS_n} q^{\charge_{\mu}(\sigma)}.
 \end{equation*}
 But the left hand side is easy to compute directly since the rows of $\lambda/\mu$ 
 occupy disjoint columns; there are $n!/((\lambda_1-\mu_1)\dotsm (\lambda_\ell-\mu_\ell))$
 ways to distribute $\{1,2,\dotsc,n\}$ in the rows of $\lambda/\mu$.
 Furthermore, in row $i$, the major index gives the Mahonian distribution $[\lambda_i-\mu_i]_q!$
 when summing over all permutations of the entries. 
 This implies the formula.
 \end{proof}

\subsection{CSP on skew specialized Macdonald polynomials}

We can generalize \cref{thm:mainCSP} to the skew setting.
We let $\phi$ act on $\COF(n \lambda / n\mu,m)$ as before, 
by cyclically shifting each block of $n$ consecutive columns one step to the right,
followed by rearranging the entries in each column such 
that a (unique) specialized Macdonald filling is obtained.
Again, $\langle \phi \rangle$ is a cyclic group of order $n$.
\begin{theorem}\label{thm:mainCSPSkew}
For every skew shape $\lambda/\mu$ and $n\geq 1$, the triple 
\[
 \left(\COF(n \lambda / n\mu,m), \langle \phi \rangle, \macdonaldE_{n \lambda / n\mu}(1^m,q,0) \right)
\]
exhibits the cyclic sieving phenomenon. 
Furthermore, the family 
\[
\{ \left(\COF(n \lambda / n\mu,m), \langle \phi \rangle, 
\macdonaldE_{n \lambda / n\mu}(1^m,q,0) \right) \}_{n=1}^\infty
\]
is Lyndon-like, as described in \cref{def:LyndonLike}.
\end{theorem}
\begin{proof}
We first note that the number of descents between two adjacent columns of the 
same height only depends on the set of entries in each column, see \cite{Uhlin2019}. 
Now consider the blocks of columns, where each block consists of $n$ consecutive columns (of the same height).
Descents involving two entries from different blocks only contribute with a multiple of $n$
to the major index.
Hence, in order to determine the major index mod $n$ of a filling, 
it suffices to examine the columns in each block separately.
Let $\nu$ be the partition, such that the parts of $\nu'$ 
is given by the multiset $\{ \lambda'_i - \mu'_i : i=1,2,\dotsc\}$.
By the previous observations,
\[
 \macdonaldE_{n \lambda / n\mu}(1^m,q,0) \equiv 
 \macdonaldE_{n \nu}(1^m,q,0) \mod (q^n-1).
\]
It is then straightforward to use the same arguments as in
the non-skew case, \cref{thm:mainCSP}, to finish the proof.
\end{proof}

Note that \cref{thm:refinedMacdonaldCSP} can be generalized to 
the skew setting using a similar argument --- we leave the details to the reader.

\section{Crystal operators on words and SSYT}\label{sec:crystalOperators}

We now recall some minimal background on crystal operators on 
words and semistandard Young tableaux, see \cite{BumpSchilling2017,Shimozono2005} for more background.

The operators $\cryse_i$, $\crysf_i : \setN^k \to \setN^k \cup \{ \emptyset \}$ are defined as follows.
Given a word $w \in \setN^k$ consider the subword $w_i$ consisting only of the letters $i$ and $i+1$.
Replace each instance of $i$ with a right bracket and each $i+1$ with a left bracket.
Remove all pairs of matching brackets and consider the remaining unmatched brackets,
which now consist of $a$ right-brackets and $b$ left-brackets.
These brackets correspond to a subword $w'$ of the form $i^a (i+1)^b$ in $w$.
The operator $\cryse_i$ acting on $w$ turns the leftmost $i+1$ of $w'$ into an $i$,
if such an entry exists, otherwise, $\cryse_i(w)\coloneqq \emptyset$.
Similarly, $\crysf_i$ acting on $w'$ turns the rightmost $i$ in $w'$ into an $i+1$,
if such an entry exists, otherwise, $\crysf_i(w)\coloneqq \emptyset$.
For example,
\[
 \cryse_1( {2, 1, 3, 1, 2, 4, 2, 1, 1, 3, 1, \underline{2}, 3, 2, 1, 2, 1} ) = 
 {2, 1, 3, 1, 2, 4, 2, 1, 1, 3, 1, \underline{1}, 3, 2, 1, 2, 1}.
\]
The operator $\cryse_i$ is a \defin{crystal raising operator}
while $\crysf_i$ is a \defin{crystal lowering operator}.
The operators also act on semistandard Young tableaux
by acting on the reading word.
We define a graph structure on words (or semistandard Young tableaux)
by having a labeled directed edge $u \to v$ with label $i$ if $f_i(u)=v$.
Examples on such components are given in \cref{fig:wordCrystal}.
\begin{figure}[!ht]
\includegraphics[scale=0.7]{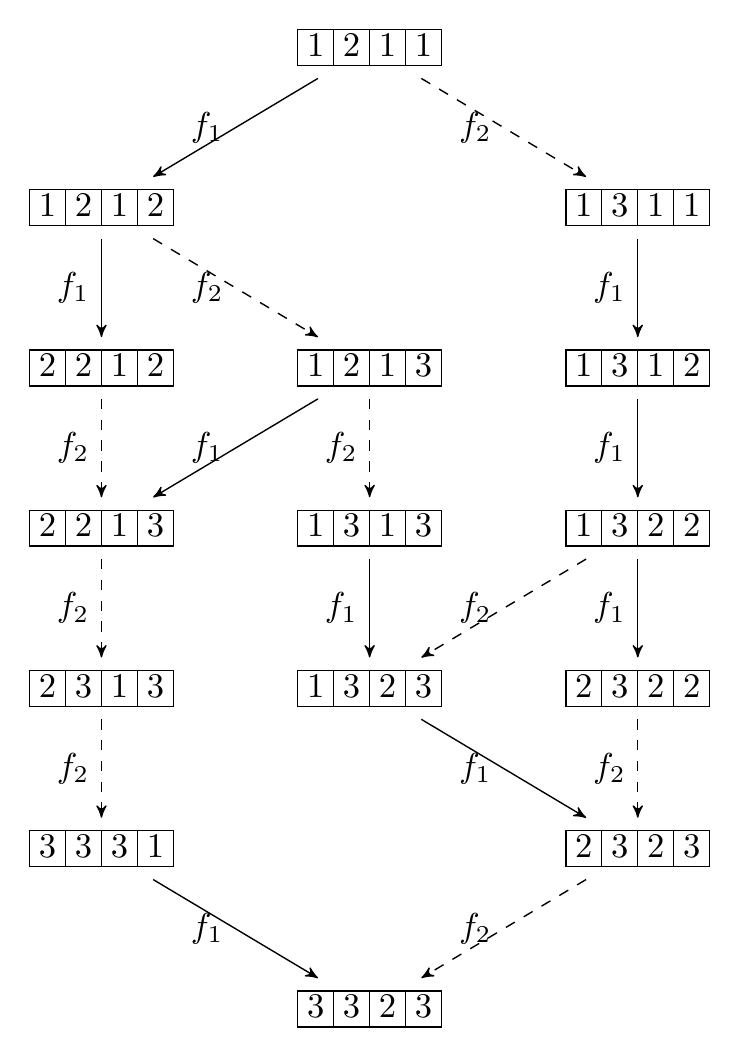}
\includegraphics[scale=0.7]{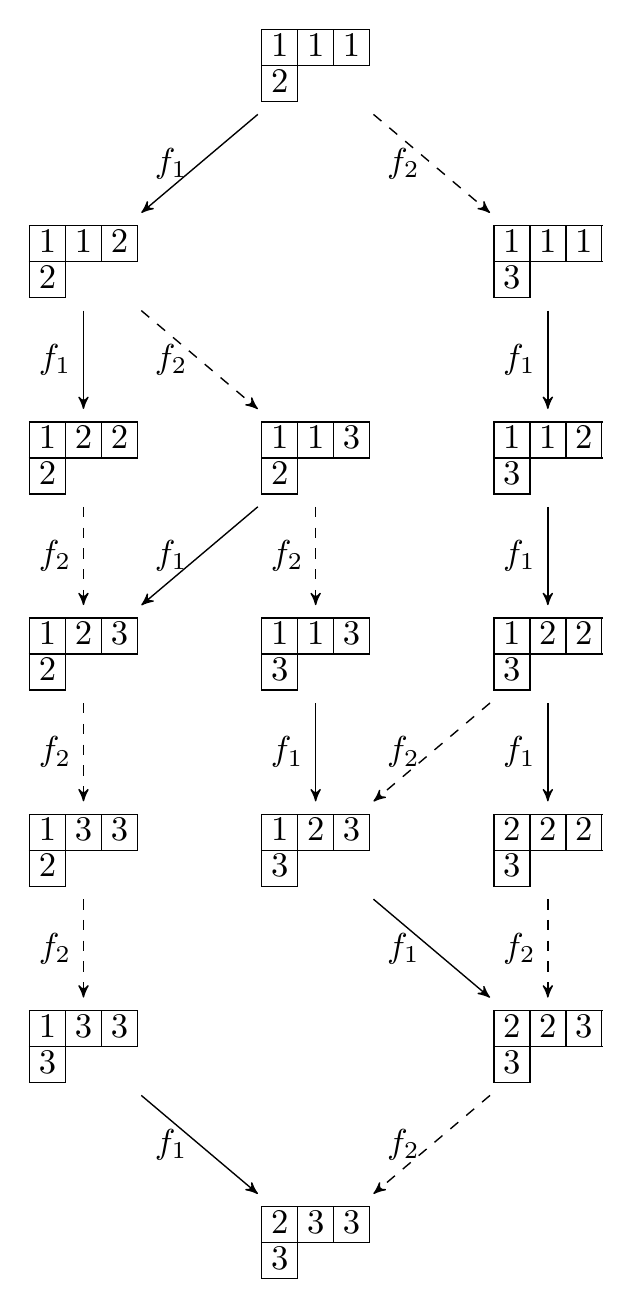}
\caption{A crystal graph on words and one on SSYT.}\label{fig:wordCrystal}
\end{figure}

\subsection{Crystal operators on COF and RSK}

We shall now define crystal operators on the set $\COF(\lambda/\mu)$.
These operators are considered in \cite{Uhlin2019},
and the non-skew case was recently considered by S.~Assaf and N.~Gonz\'ales~\cite{AssafGonzalez2018},
where the authors prove that they are indeed crystal operators.
Here, we take a slightly different route and define the operators on biwords instead ---
it is straightforward to verify that our definitions are equivalent with theirs.
These \defin{crystal biwords} are closely related to the biwords we have seen previously.

\begin{definition}[Crystal operators on fillings]
Let $F \in \COF(\lambda/\mu)$ and define the \defin{crystal biword} $\tilde{W}$,
with entry $\left(\begin{smallmatrix} j \\ c \end{smallmatrix}\right)$ appearing in $\tilde{W}$
if and only if there is a box with value $j$ in column $c$.
The entries in $\tilde{W}$ are then sorted decreasingly, primarily on the \emph{bottom row entry},
see \cref{fig:crystalBiword}. 
Note that this is map to $\tilde{W}$ is invertible if $\lambda/\mu$ is fixed.
We then define $\cryse_i(F)$, and $\crysf_i(F)$ as the result when applying 
$e_i$ and $f_i$, respectively, on the first row of $\tilde{W}$.

For the subset of coinversion-free fillings $F$ with $\maj(F)=0$, 
these operators are essentially a generalization of the raising-- and lowering 
operators defined on semistandard Young tableaux.
\end{definition}

\begin{figure}[!ht]
\ytableausetup{boxsize=0.9em}
\begin{align*}
F=\ytableaushort{\none \none 213, 331, 22,14}, \quad 
\tilde{W}=
\setlength\arraycolsep{2pt}
\begin{pmatrix}
	3 & 1 & 2 & 1 & 4 & 3 & 2 & 3 & 2 & 1\\
	5 & 4 & 3 & 3 & 2 & 2 & 2 & 1 & 1 & 1
\end{pmatrix}, \quad 
W =
\setlength\arraycolsep{2pt}
\begin{pmatrix}
	1 & 1 & 1 & 2 & 2 & 2 & 3 & 3 & 3 & 4\\
	4 & 3 & 1 & 3 & 2 & 1 & 5 & 2 & 1 & 2
\end{pmatrix} 
\end{align*}
\caption{
A skew specialized Macdonald filling $F$, the corresponding crystal biword $\tilde{W}$
and the biword $W$ used for RSK. 
Note that the only difference between $\tilde{W}$ and $W$
is the ordering of the entries.}
\label{fig:crystalBiword}
\end{figure}

The following theorem was proven independently 
by S.~Assaf and N.~Gonz\'ales~\cite{AssafGonzalez2018} and the second author \cite{Uhlin2019}.
\begin{theorem}[See \cite{Uhlin2019,AssafGonzalez2018}]\label{thm:majPres}
The operators $\cryse_i$ and $\crysf_i$ preserve major index. 
That is, suppose that $F$ and  $F'$ are in $\COF(\lambda/\mu)$ and 
that $\cryse_i(F), \crysf_i(F') \neq \emptyset$. 
Then
\[
\maj(\cryse_i(F))=\maj(F) \text{ and } \maj(\crysf_i(F'))=\maj(F').
\]
Furthermore, $F$, $\cryse_i(F)$ and $\crysf_i(F)$ differ only at boxes with entries $i$ and $i+1$.
\end{theorem}

In \cref{thm:crystalOp} below we prove that
the operators $\cryse_i$ and $\crysf_i$ define 
proper crystal graphs on the set $\COF(\lambda/\mu)$.
See \cref{fig:cofCrystalGraph} for examples of crystal graphs on coinversion-free fillings.

\begin{figure}[!ht]
\includegraphics[scale=0.7]{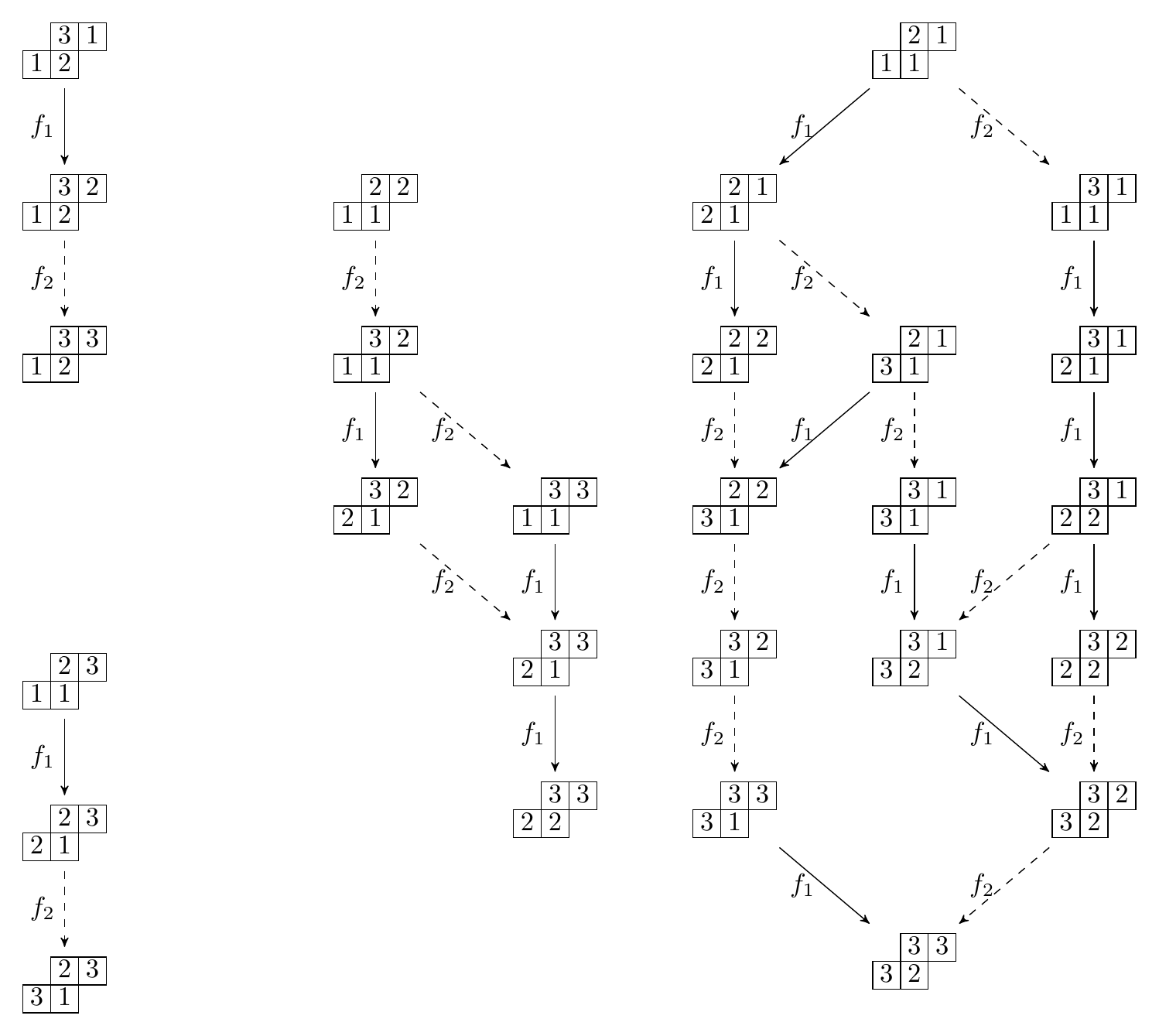}
\caption{The crystal structure on skew coinversion-free fillings of 
shape $\lambda/\mu$ with $\lambda=(3,2)$ and $\mu=(1)$ in $3$ variables.
Notice that the large component is isomorphic to 
the crystal graphs in \cref{fig:wordCrystal}.
}\label{fig:cofCrystalGraph}
\end{figure}

There is an important interaction between the crystal operators and the RSK correspondence, 
as we shall see in the following example. 
\begin{example}\label{ex:raising op biwords}
Suppose $F \in \COF(\lambda)$ with biword $W$
\[
\setlength\arraycolsep{2pt}
W = 
\begin{pmatrix}
1 & 1 & 1 & 2 & 2 & 2 & 3 & 3 & 3 & 4\\
4 & 3 & 1 & 3 & 2 & 1 & 5 & 2 & 1 & 2
\end{pmatrix}.
\]
The crystal biword is 
\[
\setlength\arraycolsep{2pt}
\tilde{W}=
\begin{pmatrix}
3 & 1 & 2 & 1 & 4 & 3 & 2& 3 & 2 &  1 \\
5 & 4 & 3 & 3 & 2 & 2 & 2& 1 & 1 &  1 
\end{pmatrix}.
\]
Apply $e_1$ on the top row; 
We have the subword $1 21 2 21$ which turn into $1 21 1 21$ under $e_1$.
\[
\setlength\arraycolsep{2pt}
\begin{pmatrix}
3 & 1 & 2 & 1 & 4 & 3 & \underline{1}& 3 & 2 &  1 \\
5 & 4 & 3 & 3 & 2 & 2 & 2& 1 & 1 &  1 
\end{pmatrix}
\]
Sort columns again to obtain the biword $W'$ that correspond to $\cryse_1(F)$.
\[
\setlength\arraycolsep{2pt}
 \cryse_1(F) \leftrightarrow 
\begin{pmatrix}
1 & 1 & 1 & 1 & 2 & 2 & 3 & 3 & 3 & 4\\
4 & 3 & 2 & 1 & 3 & 1 & 5 & 2 & 1 & 2
\end{pmatrix}.
\]
If we now perform RSK on $W$ and $W'$ we see in \cref{fig:biword RSK} 
that $\cryse_1$ has a predictable effect on the 
corresponding insertion- and recording tableau, see \cref{thm:crystalOp}.
\begin{figure}[ht!]
\[
\left(
\ytableaushort{1112,225,33,4},\quad
\ytableaushort{1234,123,13,2}
\right) \xrightarrow{e_1}
\left( \ytableaushort{1112,225,33,4}, \quad \ytableaushort{1234,123,13,1}\right)
\]
\caption{An example of the raising operator at the level of RSK. 
The pairs of tableaux correspond to the two biwords in \cref{ex:raising op biwords}.}
\label{fig:biword RSK}
\end{figure}
\end{example}

\begin{theorem}\label{thm:crystalOp}
Let $F \in \COF(\lambda/\mu)$ and suppose $\cryse_i(F) \neq \emptyset$. 
Then $\ins(F)=\ins(\cryse_i(F))$ and $e_i(\rec(F))=\rec(\cryse_i(F))$.
Stated equivalently on biwords: Let $W$ be a biword and let $\tilde{W}$ be its entries 
reordered such that it is a crystal biword. Suppose $\cryse_i(\tilde{W}) \neq \emptyset$,
and that $W'$ is the biword corresponding to $\cryse_i(\tilde{W})$.
Then $\ins(W) = \ins(W')$ and $e_i(\rec(W)) = \rec(W')$.
The analogous statements for $\crysf_i$ also hold.
\end{theorem}
\begin{proof}[Proof sketch:]
The fact that $\ins(F)=\ins(\cryse_i(F))$ follows from 
properties of the classical RSK algorithm.

The second property requires some more work, but every step is 
a routine transformation using known properties of various versions of RSK.
First restate the property to the analogous statement about the 
\emph{dual} RSK insertion algorithm (using column insertion), 
see \cite[Sec. 4.3]{Krattenthaler2006}.
One can then use that dual RSK and classical RSK are related in a 
simple manner (see e.g. \cite[Prop. 2.3.14]{Butler1994})
and reduce the problem further to the case of the classical RSK.
For classical RSK, the interaction with crystal operators is well-documented,
see e.g. \cite{Shimozono2005,Lascoux2003,BumpSchilling2017}.
\end{proof}
Using the bijection in \cref{thm:crystalOp}, we see that the set $\COF(\lambda/\mu)$
is indeed a crystal graph under the raising- and lowering operators,
since we have an equivariant bijection with crystals and crystal operators on semistandard tableaux.
We remark that S.~Assaf and N.~Gonz\'ales \cite{AssafGonzalez2018} proved 
the same result in the non-skew case by verifying the local 
characterization axioms introduced by Stembridge, see \cite{Stembridge2003}.

We shall now briefly discuss an application of the crystal operators. 
Using the crystal operators $\cryse_i$ and $\crysf_i$, 
we can define involutions on $\COF(\lambda/\mu)$. 
For a coinversion-free filling $F$, denote $m_i=m_i(F)$ the number of $i$-entries of $F$.
\begin{definition}\label{def:lsInv}
For $i \in \setN$, define the operators $\cryss_i$ on $\COF(\lambda/\mu)$ by letting
	\[ 
	\cryss_i(F) \coloneqq \begin{cases}
	(\cryse_i)^{m_{i+1}-m_i}(F) & \text{ if } m<m_{i+1}\\
	(\crysf_i)^{m-m_{i+1}}(F) & \text{ if } m_i>m_{i+1}\\
	F & \text{ if } m_i=m_{i+1}.
	\end{cases}
\]
\end{definition}
Restricted to the set of coinversion-free fillings with $\maj=0$, the operators $\cryss_i$ are 
essentially the famous Lascoux--Sch{\"{u}}tzenberger involutions \cite{LascouxSchutzenberger78}. 
The difference being that the elements with $\maj=0$ have weakly decreasing rows and strictly increasing 
columns as opposed to the weakly increasing in rows 
and strictly increasing in columns for in semistandard Young tableaux.

It is clear by \cref{thm:majPres} that the operators $\cryss_i$ are $\maj$-preserving involutions. 
Furthermore, if $F \in \COF(\lambda/\mu)$ and $\weight(F)=(w_1, \dots, w_r)$, 
then $\cryss_i(F)=(w_1, \dotsc, w_{i+1},w_i, \dotsc, w_r)$. 
This yields yet another proof that $\macdonaldE_{\lambda/\mu}(\xvec; q, 0)$ is symmetric. 
In fact, it follows from general theory of crystals that 
the operators $\cryss_1$, $\cryss_2$, $\dotsc,\cryss_{r-1}$ 
generate an $\symS_r$-action on $\COF(\lambda/\mu, r)$.

\section{Schur expansion of certain vertical-strip LLT polynomials}\label{sec:llt}

In this section, we briefly sketch that $\macdonaldE_{\lambda/\mu}(\xvec;q,0)$
sometimes is a \emph{vertical strip LLT polynomial}, up to a power of $q$.
As a consequence, we therefore obtain an explicit formula for the Schur expansion of 
these particular LLT polynomials. 
Hence, we provide a new family of LLT polynomials with a combinatorial Schur expansion,
not covered by previous results.
We note that it is a major open 
problem in general to describe the LLT polynomials in the Schur basis.

\begin{definition}[As in \cite{Haglund2005Macdonald}]
Let $\nuvec$ be a $k$-tuple of skew Young diagrams. Given such a tuple, we let 
$\SSYT(\nuvec) = \SSYT(\nuvec^1) \times \SSYT(\nuvec^2)\times \dotsm \times \SSYT(\nuvec^k)$
where $\SSYT(\lambda/\mu)$ is the set of skew semistandard Young tableaux of shape $\lambda/\mu$.
Given $T = (T^1,T^2,\dotsc,T^k) \in \SSYT(\nuvec)$,
let $\xvec^T \coloneqq \xvec^{T^1}\dotsm \xvec^{T^k}$ where $\xvec^{T^i}$
is the same monomial weight of $T^i$ as for Schur polynomials.
Given a cell $u = (r,c)$ (row, column) in a skew diagram, the \emph{content} 
is defined as $c(u)\coloneqq c-r$. 
Entries $T^i(u) > T^j(v)$ in a tuple form an \emph{inversion} if and only if
\[
 \text{$i<j$ and $c(u) = c(v)$}, \quad  \text{or} \quad 
 \text{$i>j$ and $c(u) = c(v)-1$}.
\]
The \defin{LLT polynomial} associated with the $k$-tuple $\nuvec$ is given by
\[
\LLT_\nuvec(\xvec;q) = \sum_{T \in \SSYT(\nuvec)}  q^{\inv(T)} \xvec^T 
\]
where $\inv(T)$ is the total number of inversions appearing in $T$.
One can show that $\LLT_\nuvec(\xvec;q)$ is a symmetric function, see \cite{Haglund2005Macdonald} or 
\cite{AlexanderssonPanova2016} for short proofs.
\end{definition}
LLT polynomials such that each $\nuvec^j$ is a skew shape of the form $1^a/1^b$ with $a \geq b$
are called \defin{vertical strip LLT polynomials}. 
Given a $k$-tuple $\nuvec$, we let $\mininv(\nuvec)$ be the 
minimum number of inversions obtainable over all fillings. 
That is,
\begin{equation}\label{eq:mininvDef}
 \mininv(\nuvec) \coloneqq \min_{T \in \SSYT(\nuvec)} \inv(T).
\end{equation}

\begin{example}
 A $k$-tuple of skew shapes is traditionally illustrated using the \emph{French convention}
 where box $(1,1)$ of each shape $\nuvec^i$ is placed on the line $y=x$ with content $0$,
 and Cartesian coordinates are used. Below, we illustrate an element 
 \[
 T \in \SSYT(1^3/\emptyset)\times \SSYT(1^3/1^1) \times \SSYT(1^2/1^1)\times \SSYT(1^3/\emptyset)
 \]
 which appears when computing the vertical-strip LLT polynomial $\LLT_\nuvec(\xvec;q)$
 for $\nuvec = (111/\emptyset,111/1, 11/1, 111/\emptyset)$.
 \[
 T=
\begin{tikzpicture}[baseline=(current bounding box.center)]
\draw[step=1em, gray, very thin] (-0.001,0) grid (9em,9em);
\draw[gray, very thin, dashed,x=1em,y=1em] (0,0) -- (9,9);
\draw[gray, very thin, dashed,x=1em,y=1em] (0,1) -- (8,9);
\draw[gray, very thin, dashed,x=1em,y=1em] (0,2) -- (7,9);
\node[x=1em,y=1em] (9) at (6.5, 6.5) {2};
\node[x=1em,y=1em] (8) at (0.5, 0.5) {1};
\node[x=1em,y=1em] (7) at (6.5, 7.5) {4};
\node[x=1em,y=1em] (6) at (2.5, 3.5) {3};
\node[x=1em,y=1em] (5) at (0.5, 1.5) {2};
\node[x=1em,y=1em] (4) at (6.5, 8.5) {5};
\node[x=1em,y=1em] (3) at (4.5, 5.5) {5};
\node[x=1em,y=1em] (2) at (2.5, 4.5) {6};
\node[x=1em,y=1em] (1) at (0.5, 2.5) {4};
\end{tikzpicture}
 \]
There are two inversions involving boxes $u$ and $v$ where $c(u)=c(v)$
and six inversions for which $c(u)=c(v)-1$. 
Hence, $T$ contributes with $q^8 x_1 x_2^2 x_3x_4^2 x_5^2 x_6$.
The full LLT polynomial $\LLT_\nuvec(\xvec;q)$ in the Schur basis is given by 
\begin{align*}
 &q^8 \schurS_{333}+q^7 \schurS_{432}+(q^9+q^{10}+q^{11}) \schurS_{3222}+(q^8+2 q^9+q^{10}) \schurS_{3321} +(q^8+q^9) \schurS_{4221} \\
 &+q^8 \schurS_{4311}+(q^{10}+q^{11}+q^{12}+q^{13}) \schurS_{22221} +(q^9+3 q^{10}+2 q^{11} +q^{12}) \schurS_{32211} \\
 &+(q^9+q^{10}+q^{11}) \schurS_{33111} +(q^9+q^{10}) \schurS_{42111} +(2 q^{11}+2 q^{12}+q^{13}+q^{14}) \schurS_{222111} \\
 &+(q^{10}+2 q^{11}+2 q^{12}+q^{13}) \schurS_{321111} +q^{11} \schurS_{411111} +(q^{12}+2 q^{13}+q^{14}+q^{15}) \schurS_{2211111} \\
 &+(q^{12}+q^{13}+q^{14}) \schurS_{3111111} +(q^{14}+q^{15}+q^{16}) \schurS_{21111111}+q^{17} \schurS_{111111111}.
\end{align*}
As $q^7$ is the lowest power of $q$ that appear in the expansion, we must have that $\mininv(\nuvec)=7$.
\end{example}

The current state-of-the-art regarding combinatorial proofs of Schur positivity 
of LLT polynomials is as follows.
\begin{itemize}
 \item When all shapes in $\nuvec$ are non-skew, the coefficients in the Schur basis are known to be 
 certain parabolic Kazhdan--Lusztig polynomials, see \cite{LeclercThibon2000}.
 Hence, the coefficients are in $\setN[q]$.
 In particular, this case contains the Hall--Littlewood symmetric functions.
 
 \item Whenever the $k$-tuple of shapes $\nuvec$ consists of at most $3$ shapes, 
 all avoiding an arrangement of $2\times 2$-boxes (that is, they are ribbons), 
 Schur positivity is given by a combinatorial formula, see J.~Blasiak~\cite{Blasiak2016}.
 
 \item A few other cases when each shape in $\nuvec$ is a single box is given in \cite{HuhNamYoo2018}.
\end{itemize}

\begin{theorem}\label{thm:lltNewFormula}
Let $\lambda/\mu$ be a skew shape such that no column contains more than two boxes. 
Then 
\[
 \macdonaldE_{\lambda'/\mu'}(\xvec;q,0) = q^{-\mininv(\nuvec)} \LLT_{\nuvec}(\xvec;q)
\]
where $\nuvec_{j}$ is the vertical strip $(\lambda_j) /(\mu_j)$.
\end{theorem}
\begin{proof}[Proof sketch]
For the \defin{modified Macdonald polynomials} $\macdonaldH_{\lambda}(\xvec;q,t)$,
we have the symmetry $\macdonaldH_{\lambda}(\xvec;q,t) = \macdonaldH_{\lambda'}(\xvec;t,q)$.
This interchanges the r\^ole of inversion triples and major index, see \cite{qtCatalanBook}.
This relationship extends to modified Macdonald polynomials indexed by 
skew shapes $\lambda/\mu$ as long as each column contains at most two boxes, 
see J.~Brandlow,~\cite[Thm. 5]{Brandlow2007}.

There is a correspondence between inversion triples and inversions
that appearing definition of LLT polynomials.
In \cite[Eq. (23)]{Haglund2005Macdonald}, the authors 
provide (via a straightforward bijective argument) an expansion of the form
\begin{align}\label{eq:macdonaldHInLLT}
 \macdonaldH_{\lambda}(\xvec;q,t) = \sum_{D}  q^{\maj(D)} t^{-\stat(D)} \LLT_{\nu(D)}(\xvec;t),
\end{align}
where the sum runs over all subsets (possible descents) of boxes $(i,j)$ with $i>1$ of the diagram $\lambda$.
In particular, the coefficient of the terms maximizing the major index a vertical-strip LLT polynomial.
This expansion has a natural extension to skew shapes and one can check that $\stat(\cdot)$ corresponds to $\mininv(\cdot)$
for the highest-degree term.
Combining all these observations we have
\begin{align*}
 \macdonaldE_{\lambda'/\mu'}(\xvec;q,0) =
 [t^\ast]\macdonaldH_{\lambda'/\mu'}(\xvec;t,q)=
 [t^\ast]\macdonaldH_{\lambda/\mu}(\xvec;q,t) =
 q^{-\mininv(\nuvec)}\LLT_{\nuvec}(\xvec;q).
\end{align*}
The first identity is due to \eqref{eq:macdonaldEAsSkewMacdonaldH}.
The second identity is the tricky part and relies on \cite{Brandlow2007}.
The third identity is a simple generalization of \eqref{eq:macdonaldHInLLT}.
\end{proof}

\begin{example}
We illustrate \cref{thm:lltNewFormula} in the case $\lambda/\mu = 4431/31$.
The skew shape $\lambda/\mu$ is illustrated in \eqref{eq:lltExample} where we have labeled the boxes from right to left,
top to bottom. The corresponding $k$-tuple of vertical strips is shown to the right.
The labeling has the property that it maps \emph{inversion pairs} in the filling to the right,
to inversions in the LLT diagram, see \cite{Haglund2005Macdonald} for details.

\begin{equation}\label{eq:lltExample}
 \ytableaushort{{\none}{\none}{\none}1,{\none}532,764,8}
 \qquad 
 \qquad 
 \begin{tikzpicture}[baseline=(current bounding box.center)]
\draw[step=1em, gray, very thin] (-0.001,-3em) grid (11em,8em);
\draw[gray, very thin, dashed,x=1em,y=1em] (0,0) -- (8,8);
\draw[gray, very thin, dashed,x=1em,y=1em] (0,-1) -- (9,8);
\draw[gray, very thin, dashed,x=1em,y=1em] (0,-2) -- (10,8);
\draw[gray, very thin, dashed,x=1em,y=1em] (0,-3) -- (11,8);
\node[x=1em,y=1em] (1) at (0.5, 0.5) {1};
\node[x=1em,y=1em] (2) at (3.5, 3.5) {2};
\node[x=1em,y=1em] (3) at (3.5, 2.5) {3};
\node[x=1em,y=1em] (5) at (3.5, 1.5) {5};
\node[x=1em,y=1em] (4) at (7.5, 6.5) {4};
\node[x=1em,y=1em] (6) at (7.5, 5.5) {6};
\node[x=1em,y=1em] (7) at (7.5, 4.5) {7};
\node[x=1em,y=1em] (7) at (10.5, 7.5) {8};
\end{tikzpicture}
\end{equation}
Notice that no column contains more than two boxes so the conditions in the theorem applies.
The $k$-tuple $\nuvec$ is $1111/111$, $1111/1$, $111/\emptyset$, $1/\emptyset$, 
and it is easy (for a computer) to check that
\begin{align*}
\macdonaldE_{\lambda'/\mu'}(\xvec;q,0) &=  
\schurS_{332}+\schurS_{422}+(1+q^2) \schurS_{2222}+(2+2 q) \schurS_{3221}+\schurS_{3311}+\schurS_{4211} \\
&+(3 q+q^2) \schurS_{22211}+4 q \schurS_{32111}+q \schurS_{41111}+4 q^2 \schurS_{221111}+3 q^2 \schurS_{311111} \\
&+3 q^3 \schurS_{2111111}+q^4 \schurS_{11111111} 
\end{align*}
and that this is also equal to $q^{-1} \LLT_{\nuvec}(\xvec;q)$.

As a final check, we verify one of the coefficients with the combinatorial formula.
Using the notation in \cref{thm:skewEInSchurExpansion}, $\alpha = 1331$.
The term $(2+2q) \schurS_{3221}$ then arises from the four semistandard tableaux
\[
\substack{\ytableaushort{1222,333,4}, \\ 1} \quad
\substack{\ytableaushort{1223,233,4}, \\ 0} \quad
\substack{\ytableaushort{1223,234,3}, \\0 } \quad
\substack{\ytableaushort{1224,233,3}. \\1 }
\]
where the value of $\charge_{31}(w)=\charge(w \cdot 2111)$ is shown under each tableau.
\end{example}

\subsection*{Acknowledgements}

The authors would like to thank Svante Linusson and Samu Potka for helpful discussions.
We also thank Jim Haglund for suggesting to look at the connection with LLT polynomials 
and the relevance of \cite{Brandlow2007}.
The first author is funded by the Swedish Research Council (Vetenskapsrådet), grant 2015-05308.

\section{Appendix: How to compute Kostka--Foulkes polynomials}\label{sec:kostkaFoulkes}

We shall briefly describe how to compute the 
coefficients $K_{\lambda\mu}(q)$ appearing in \cref{eq:kostkaFoulkes}.
This combinatorial model was first described by A. Lascoux and M. Schützenberger
in \cite{LascouxSchutzenberger78}.
For a permutation $\sigma \in \symS_k$, let $\Des(\sigma) \coloneqq \{i \in [k-1] : \sigma_{i+1}<\sigma_{i} \}$,
and let $\rev(\sigma)$ be the reverse, $\sigma_n,\sigma_{k-1},\dotsc,\sigma_1$.
We can now introduce the notion of \defin{charge} of a permutation.
\begin{equation}
 \charge(\sigma) \coloneqq \maj(\rev(\sigma^{-1})) = 
 \sum_{i \notin \Des(\sigma^{-1})} (k-i).
\end{equation}
For example, 
\[
\charge(198423765) = \maj(\rev(156498732) ) = \maj(237894651) = 20.
\]

Given a word $w$ with content $\mu \vdash n$,
we partition its entries into \defin{standard subwords} as follows.
Start from the right of $w$ and mark the first occurrence of $1$.
Proceed to the left, and mark the first occurrence of $2$,
then $3$ and so on, wrapping around the end if nessecary,
until $\mu'_1$ entries have been marked.
This subword is the first standard subword of $w$.
Remove this subword, and repeat the process to find the second standard subword,
of length $\mu'_2$.

For example, the first standard subword in $w = 2 1 1 2 3 5 4 3 4 1 1 2 2 3$
has been circled.
\[
2, 1, 1, \circled{2}, 3, \circled{5}, 4, 3, \circled{4}, 1, \circled{1}, 2, 2, \circled{3}.
\]
In total, we have four standard subwords in $w$, with corresponding charge values
\[
 \charge(25413) = 3,\quad 
 \charge(2431)=2 \quad 
 \charge(132)=2  \quad 
 \charge(12) = 1,
\]
and we define $\charge(w)$ as the sum of the charge values of the standard subwords.
In the example above,  $\charge(w) = 8$.

Recall the definition of the reading word $\rw(T)$ of a semistandard Young tableau 
from \cref{def:SSYT}. We then define $\charge(T) \coloneqq \charge(\rw(T))$
and the \defin{Kostka--Foulkes polynomial} $K_{\lambda\mu}(q)$ may be computed as
\begin{equation}\label{eq:kostkaFoulkesDef}
 K_{\lambda\mu}(q) = \sum_{T \in \SSYT(\lambda,\mu)} q^{\charge(T)}.
\end{equation}

\begin{example}[Computing a Kostka--Foulkes polynomial]
Consider the case $\lambda=421$, $\mu=3211$. There are four tableaux in $\SSYT(\lambda,\mu)$.
Below, these are displayed, each with the list of standard subwords
and corresponding charge values.
\begin{equation}
 \substack{
 \ytableaushort{1114,22,3} \\ 3214,\; 21,\; 1 \\ 1+0+0
 }
 \quad
\substack{
 \ytableaushort{1113,22,4} \\ 4213,\; 21,\; 1 \\ 2+0+0
 }
 \quad
 \substack{
 \ytableaushort{1112,24,3} \\ 3241,\; 12,\; 1 \\ 1+1+0
 }
 \quad
 \substack{
 \ytableaushort{1112,23,4} \\ 4231,\; 12,\; 1 \\ 2+1+0
 }
\end{equation}
Hence, $K_{\lambda\mu}(q) = q+2q^2+q^3$.
\end{example}

We remark that there is also the notion of \defin{skew Kostka--Foulkes polynomials},
$K_{\lambda/\mu,\nu}(q)$ see \cite{Butler1994}. 
They are defined as in \eqref{eq:kostkaFoulkesDef}, 
where the sum now ranges over the elements in $\SSYT(\lambda/\mu,\nu)$.
To our knowledge, there is no obvious relationship between these skew Kostka--Foulkes polynomials
and the polynomials
\[
 \sum_{T \in \SSYT(\nu',\alpha)} q^{\charge_{\mu'}(T)} 
\]
appearing in \eqref{eq:skewEInSchurExpansion}.

\bibliographystyle{alphaurl}
\bibliography{bibliography}

\begin{thebibliography}{Rho10b}

\bibitem[AG18]{AssafGonzalez2018}
Sami Assaf and Nicolle~S. Gonz{\'{a}}lez.
\newblock Crystal graphs, key tabloids, and nonsymmetric {M}acdonald
  polynomials.
\newblock In {\em 30th {I}nternational Conference on Formal Power Series and
  Algebraic Combinatorics ({H}anover)}, volume 80B. S{\'{e}}minaire
  Lotharingien de Combinatoire, 2018.
\newblock 12 pages.
\newblock URL: \url{https://www.mat.univie.ac.at/\textasciitilde
  slc/wpapers/FPSAC2018/81-Assaf-Gonzalez.html}.

\bibitem[Ale15]{Alexandersson2015gbMacdonald}
Per Alexandersson.
\newblock Non-symmetric {M}acdonald polynomials and {D}emazure--{L}usztig
  operators.
\newblock {\em ArXiv e-prints}, 2015.
\newblock \href {http://arxiv.org/abs/1602.05153} {\path{arXiv:1602.05153}}.

\bibitem[ALP19]{AlexanderssonLinussonPotka2019}
Per Alexandersson, Svante Linusson, and Samu Potka.
\newblock The cyclic sieving phenomenon on circular {D}yck paths.
\newblock {\em ArXiv e-prints}, 2019.
\newblock \href {http://arxiv.org/abs/1903.01327} {\path{arXiv:1903.01327}}.

\bibitem[AP18]{AlexanderssonPanova2016}
Per Alexandersson and Greta Panova.
\newblock {LLT} polynomials, chromatic quasisymmetric functions and graphs with
  cycles.
\newblock {\em Discrete Mathematics}, 341(12):3453--3482, December 2018.
\newblock \href {http://dx.doi.org/10.1016/j.disc.2018.09.001}
  {\path{doi:10.1016/j.disc.2018.09.001}}.

\bibitem[AS17]{AlexanderssonSawhney2017}
Per Alexandersson and Mehtaab Sawhney.
\newblock A major-index preserving map on fillings.
\newblock {\em Electronic Journal of Combinatorics}, 24(4):1--30, 2017.
\newblock URL:
  \url{http://www.combinatorics.org/ojs/index.php/eljc/article/view/v24i4p3}.

\bibitem[AS19]{AlexanderssonSawhney2019}
Per Alexandersson and Mehtaab Sawhney.
\newblock Properties of non-symmetric {M}acdonald polynomials at $q=1$ and
  $q=0$.
\newblock {\em Annals of Combinatorics}, 23(2):219--239, May 2019.
\newblock \href {http://dx.doi.org/10.1007/s00026-019-00432-z}
  {\path{doi:10.1007/s00026-019-00432-z}}.

\bibitem[Ass18]{Assaf2018Kostka}
Sami Assaf.
\newblock Nonsymmetric {M}acdonald polynomials and a refinement of
  {K}ostka--{F}oulkes polynomials.
\newblock {\em Transactions of the American Mathematical Society},
  370(12):8777--8796, July 2018.
\newblock \href {http://dx.doi.org/10.1090/tran/7374}
  {\path{doi:10.1090/tran/7374}}.

\bibitem[Bla16]{Blasiak2016}
Jonah Blasiak.
\newblock Haglund's conjecture on 3-column {M}acdonald polynomials.
\newblock {\em Mathematische Zeitschrift}, 283(1-2):601--628, January 2016.
\newblock \href {http://dx.doi.org/10.1007/s00209-015-1612-7}
  {\path{doi:10.1007/s00209-015-1612-7}}.

\bibitem[Bra07]{Brandlow2007}
Jason Brandlow.
\newblock {\em Combinatorics of {M}acdonald polynomials and extensions}.
\newblock PhD thesis, UC San Diego, 2007.
\newblock URL: \url{https://escholarship.org/uc/item/5zd262sp}.

\bibitem[BS17]{BumpSchilling2017}
Daniel Bump and Anne Schilling.
\newblock {\em Crystal Bases: representations and combinatorics}.
\newblock {W}orld {S}cientific, October 2017.
\newblock \href {http://dx.doi.org/10.1142/9876} {\path{doi:10.1142/9876}}.

\bibitem[But94]{Butler1994}
Lynne~M. Butler.
\newblock {\em Subgroup Lattices and Symmetric Functions}.
\newblock American Mathematical Society, 1994.
\newblock URL: \url{https://bookstore.ams.org/memo-112-539}.

\bibitem[Che95]{Cherednik1995nonsymmetric}
Ivan Cherednik.
\newblock Nonsymmetric {M}acdonald polynomials.
\newblock {\em International Mathematics Research Notices}, 1995(10):483, 1995.
\newblock \href {http://dx.doi.org/10.1155/s1073792895000341}
  {\path{doi:10.1155/s1073792895000341}}.

\bibitem[DLT94]{DesarmenienLeclercThibon1994}
J.~D{\'{e}}sarm{\'{e}}nien, B.~Leclerc, and J.-Y. Thibon.
\newblock Hall-{L}ittlewood functions and {K}ostka--{F}oulkes polynomials in
  representation theory.
\newblock {\em S{\'{e}}minaire Lotharingien de Combinatoire [electronic only]},
  32:38, 1994.
\newblock URL: \url{http://eudml.org/doc/119019}.

\bibitem[Gor19]{Gorodetsky2019}
Ofir Gorodetsky.
\newblock {$q$}-congruences, with applications to supercongruences and the
  cyclic sieving phenomenon.
\newblock {\em International Journal of Number Theory}, May 2019.
\newblock \href {http://dx.doi.org/10.1142/s1793042119501069}
  {\path{doi:10.1142/s1793042119501069}}.

\bibitem[Hag07]{qtCatalanBook}
James Haglund.
\newblock {\em The $q,t$-{C}atalan numbers and the space of diagonal harmonics
  ({U}niversity lecture series)}.
\newblock American Mathematical Society, 2007.
\newblock URL: \url{https://www.math.upenn.edu/\textasciitilde
  jhaglund/books/qtcat.pdf}.

\bibitem[HHL05]{Haglund2005Macdonald}
James Haglund, Mark Haiman, and Nicholas Loehr.
\newblock A combinatorial formula for {M}acdonald polynomials.
\newblock {\em J. Amer. Math. Soc.}, 18(03):735--762, July 2005.
\newblock \href {http://dx.doi.org/10.1090/s0894-0347-05-00485-6}
  {\path{doi:10.1090/s0894-0347-05-00485-6}}.

\bibitem[HHL08]{HaglundHaimanLoehr2008}
James Haglund, Mark Haiman, and Nick Loehr.
\newblock A combinatorial formula for nonsymmetric {M}acdonald polynomials.
\newblock {\em American Journal of Mathematics}, 130(2):359--383, 2008.
\newblock \href {http://dx.doi.org/10.1353/ajm.2008.0015}
  {\path{doi:10.1353/ajm.2008.0015}}.

\bibitem[HNY18]{HuhNamYoo2018}
JiSun Huh, Sun-Young Nam, and Meesue Yoo.
\newblock Melting lollipop chromatic quasisymmetric functions and {S}chur
  expansion of unicellular {LLT} polynomials.
\newblock {\em ArXiv e-prints}, 2018.
\newblock \href {http://arxiv.org/abs/1812.03445} {\path{arXiv:1812.03445}}.

\bibitem[Kra06]{Krattenthaler2006}
Christian Krattenthaler.
\newblock Growth diagrams, and increasing and decreasing chains in fillings of
  {F}errers shapes.
\newblock {\em Advances in Applied Mathematics}, 37(3):404--431, September
  2006.
\newblock \href {http://dx.doi.org/10.1016/j.aam.2005.12.006}
  {\path{doi:10.1016/j.aam.2005.12.006}}.

\bibitem[KS97]{KnopSahi1997}
Friedrich Knop and Siddhartha Sahi.
\newblock A recursion and a combinatorial formula for {J}ack polynomials.
\newblock {\em Inventiones Mathematicae}, 128(1):9--22, March 1997.
\newblock \href {http://dx.doi.org/10.1007/s002220050134}
  {\path{doi:10.1007/s002220050134}}.

\bibitem[KSS01]{KirillovSchillingShimozono2001}
Anatol~N. Kirillov, Anne Schilling, and Mark Shimozono.
\newblock Various representations of the generalized {K}ostka polynomials.
\newblock In {\em The Andrews Festschrift}, pages 209--226. Springer Berlin
  Heidelberg, 2001.
\newblock \href {http://dx.doi.org/10.1007/978-3-642-56513-7_10}
  {\path{doi:10.1007/978-3-642-56513-7_10}}.

\bibitem[Las03]{Lascoux2003}
Alain Lascoux.
\newblock Double crystal graphs.
\newblock In {\em Studies in Memory of Issai Schur}, pages 95--114.
  Birkh{\"{a}}user Boston, 2003.
\newblock \href {http://dx.doi.org/10.1007/978-1-4612-0045-1_5}
  {\path{doi:10.1007/978-1-4612-0045-1_5}}.

\bibitem[LLT94]{LascouxLeclercThibon1994}
Alain Lascoux, Bernard Leclerc, and Jean-Yves Thibon.
\newblock Green polynomials and {H}all-{L}ittlewood functions at roots of
  unity.
\newblock {\em European Journal of Combinatorics}, 15(2):173--180, March 1994.
\newblock \href {http://dx.doi.org/10.1006/eujc.1994.1019}
  {\path{doi:10.1006/eujc.1994.1019}}.

\bibitem[LS78]{LascouxSchutzenberger78}
Alain Lascoux and Marcel-Paul Sch{\"{u}}tzenberger.
\newblock Sur une conjecture de {H}. {O}. {F}oulkes.
\newblock {\em C. R. Acad. Sci. Paris S{\'{e}}r. A-B}, 286(7):A323--A324, 1978.

\bibitem[LT00]{LeclercThibon2000}
Bernard Leclerc and Jean-Yves Thibon.
\newblock {L}ittlewood--{R}ichardson coefficients and {K}azhdan--{L}usztig
  polynomials.
\newblock In {\em Combinatorial Methods in Representation Theory}, volume~28,
  pages 155--220. Mathematical Society of Japan, 2000.
\newblock \href {http://dx.doi.org/10.2969/aspm/02810155}
  {\path{doi:10.2969/aspm/02810155}}.

\bibitem[Mac95]{Macdonald1995}
Ian~G. Macdonald.
\newblock {\em Symmetric functions and {H}all polynomials}.
\newblock Oxford Mathematical Monographs. The Clarendon Press Oxford University
  Press, New York, second edition, 1995.
\newblock With contributions by A. Zelevinsky, Oxford Science Publications.

\bibitem[Opd95]{Opdam1995}
Eric~M. Opdam.
\newblock Harmonic analysis for certain representations of graded {H}ecke
  algebras.
\newblock {\em Acta Mathematica}, 175(1):75--121, 1995.
\newblock \href {http://dx.doi.org/10.1007/BF02392487}
  {\path{doi:10.1007/BF02392487}}.

\bibitem[PPR08]{PetersenPylyavskyyRhoades2008}
T.~Kyle Petersen, Pavlo Pylyavskyy, and Brendon Rhoades.
\newblock Promotion and cyclic sieving via webs.
\newblock {\em Journal of Algebraic Combinatorics}, 30(1):19--41, September
  2008.
\newblock \href {http://dx.doi.org/10.1007/s10801-008-0150-3}
  {\path{doi:10.1007/s10801-008-0150-3}}.

\bibitem[Rho10a]{Rhoades2010}
Brendon Rhoades.
\newblock Cyclic sieving, promotion, and representation theory.
\newblock {\em Journal of Combinatorial Theory, Series A}, 117(1):38--76,
  January 2010.
\newblock \href {http://dx.doi.org/10.1016/j.jcta.2009.03.017}
  {\path{doi:10.1016/j.jcta.2009.03.017}}.

\bibitem[Rho10b]{Rhoades2010b}
Brendon Rhoades.
\newblock {H}all--{L}ittlewood polynomials and fixed point enumeration.
\newblock {\em Discrete Mathematics}, 310(4):869--876, February 2010.
\newblock \href {http://dx.doi.org/10.1016/j.disc.2009.10.003}
  {\path{doi:10.1016/j.disc.2009.10.003}}.

\bibitem[RSW04]{ReinerStantonWhite2004}
Victor Reiner, Dennis Stanton, and Dennis White.
\newblock The cyclic sieving phenomenon.
\newblock {\em Journal of Combinatorial Theory, Series A}, 108(1):17--50,
  October 2004.
\newblock \href {http://dx.doi.org/10.1016/j.jcta.2004.04.009}
  {\path{doi:10.1016/j.jcta.2004.04.009}}.

\bibitem[Sag92]{Sagan1992}
Bruce~E. Sagan.
\newblock Congruence properties of $q$-analogs.
\newblock {\em Advances in Mathematics}, 95(1):127--143, September 1992.
\newblock \href {http://dx.doi.org/10.1016/0001-8708(92)90046-n}
  {\path{doi:10.1016/0001-8708(92)90046-n}}.

\bibitem[Sag11]{Sagan2011}
Bruce~E. Sagan.
\newblock The cyclic sieving phenomenon: a survey.
\newblock In Robin Chapman, editor, {\em Surveys in Combinatorics 2011}, pages
  183--234. Cambridge University Press, 2011.
\newblock \href {http://dx.doi.org/10.1017/cbo9781139004114.006}
  {\path{doi:10.1017/cbo9781139004114.006}}.

\bibitem[Shi]{Shimozono2005}
Mark Shimozono.
\newblock Crystals for dummies.
\newblock Online.
\newblock URL: \url{https://www.aimath.org/WWN/kostka/crysdumb.pdf}.

\bibitem[Slo19]{OEIS}
Neil J.~A. Sloane.
\newblock The {O}n-{L}ine {E}ncyclopedia of {I}nteger {S}equences, 2019.
\newblock URL: \url{https://oeis.org}.

\bibitem[Sta01]{StanleyEC2}
Richard~P. Stanley.
\newblock {\em Enumerative {C}ombinatorics: {V}olume 2}.
\newblock Cambridge University Press, 1st edition, 2001.
\newblock URL: \url{http://www.worldcat.org/isbn/0521789877}.

\bibitem[Ste03]{Stembridge2003}
John~R. Stembridge.
\newblock A local characterization of simply-laced crystals.
\newblock {\em Transactions of the American Mathematical Society},
  355(12):4807--4823, 2003.
\newblock \href {http://dx.doi.org/10.1090/S0002-9947-03-03042-3}
  {\path{doi:10.1090/S0002-9947-03-03042-3}}.

\bibitem[SW00]{ShimozonoWeyman2000}
Mark Shimozono and Jerzy Weyman.
\newblock Graded characters of modules supported in the closure of a nilpotent
  conjugacy class.
\newblock {\em European Journal of Combinatorics}, 21(2):257--288, February
  2000.
\newblock \href {http://dx.doi.org/10.1006/eujc.1999.0344}
  {\path{doi:10.1006/eujc.1999.0344}}.

\bibitem[TZ03]{TudoseZabrocki2003}
Geanina Tudose and Mike Zabrocki.
\newblock A q-analog of {S}chur's {Q}-functions.
\newblock In Naihuan Jing, editor, {\em Algebraic Combinatorics and Quantum
  Groups}. World Scientific, 2003.
\newblock \href {http://dx.doi.org/10.1142/5331} {\path{doi:10.1142/5331}}.

\bibitem[Uhl19]{Uhlin2019}
Joakim Uhlin.
\newblock Combinatorics of {M}acdonald polynomials and cyclic sieving.
\newblock Master's thesis, KTH, Mathematics (Div.), January 2019.
\newblock URL:
  \url{http://kth.diva-portal.org/smash/record.jsf?pid=diva2%3A1282825}.

\end{thebibliography}

\end{document}